\documentclass[10pt]{article}
\usepackage[a4paper, left=3cm, right=3cm, height=22cm]{geometry}
\usepackage{amsmath,amsfonts,amsthm,bbm, amssymb}
\usepackage[arrow,matrix]{xy}
\usepackage{enumerate}

\usepackage[usenames]{color}
\usepackage{amssymb}
\usepackage{ulem}
\usepackage[inline]{showlabels}

\theoremstyle{plain}
\numberwithin{equation}{section}
\newtheorem{thm}{Theorem}[section]
\newtheorem{cor}[thm]{Corollary}
\newtheorem{lem}[thm]{Lemma}
\newtheorem{prop}[thm]{Proposition}

\theoremstyle{definition}
\newtheorem{df}[thm]{Definition}
\newtheorem{exe}[thm]{Example}
\newtheorem{rmk}[thm]{Remark}

\newtheorem{con}[thm]{Convention}
\newtheorem{cl}[thm]{Claim}
	
\usepackage[explicit]{titlesec}

\titleformat{\subsubsection}
  [runin] 
  {\normalfont\normalsize\bfseries} 
  {\thesubsubsection.} 
  {0.5em} 
	{#1} 

\newcommand{\mb}{\mathbb}
\newcommand{\mf}{\mathfrak}
\newcommand{\ml}{\mathcal}

\newcommand{\V}{\mathbf{F}}

\newcommand{\G}{G}
\newcommand{\lw}{\langle} 
\newcommand{\rw}{\vert} 
\newcommand{\law}{\vert} 
\newcommand{\raw}{\rangle} 

\usepackage[pdfauthor={Jin Cao}, %
				pdftitle={Cdga}%
			dvips,colorlinks=true]{hyperref}
\hypersetup{
	bookmarksnumbered=true,
	linkcolor=black,
}
\newcounter{elno}

\begin{document}
\author{Jin Cao}
\title{Differential Graded Algebras Over Some Reductive Group}

\maketitle

\begin{abstract}
In this paper, we study the general properties of commutative differential graded algebras in the category of representations over a reductive algebraic group with an injective central cocharacter. Besides describing the derived category of differential graded modules over such an algebra, we also provide a criterion for the existence of a t-structure on the derived category together with a characterization of the coordinate ring of the Tannakian fundamental group of its heart.
\end{abstract}



\setcounter{tocdepth}{1}
\tableofcontents

\section{Introduction}
In \cite{KM}, Kriz and May develop a general theory about Adams graded commutative differential algebra (cdga) over $\mb{Q}$. More precisely, given an Adams cdga $A$, they consider the bounded derived  category $\ml{D}^f_{A}$ of Adams graded dg $A$-modules and show a lot of formal properties of $\ml{D}^f_{A}$. Assume that $A$ is cohomologically connected, which says that negative cohomological degree parts of $A$ vanish and zero-th cohomological degree part of $A$ is isomorphic to $\mb{Q}$. Then they show that $\ml{D}^f_{A}$ has a $t$-structure. Furthermore, they can use the reduced bar construction (Section \ref{bar con}) to describe the heart $\ml{H}^f_{A}$ of $\ml{D}^f_{A}$. We recall that the reduced bar construction $\bar{B}(A)$ is a differential graded Hopf algebra. Taking the first cohomological part will get a Hopf algebra $H^0(\bar{B}(A))$, which corresponds to a pro-affine group scheme $G_A$ over $\mb{Q}$ with $\mb{G}_m$ action. Then $\ml{H}^f_A$ is equivalent to the category of graded representations of $G_A$ in finite dimensional $\mb{Q}$ vector spaces.
\

Let us mention one application of the theory of Adams cdgas. Given any field $k$, Kriz and May take $A$ to be Bloch's cycle complex $\ml{N}_k$. When $\ml{N}_k$ is cohomologically connected (depending on $k$), $\ml{H}_{\ml{N}_k}^f$ will exist and coincide with an earlier construction of mixed Tate motives by Bloch and Kriz \cite{BK}. For the definition of Bloch's cycle complex $\ml{N}_k$, we refer to \cite{L}. Later Spitzweck \cite{S} defines an equivalence 
$$\theta_k: \ml{D}^f_{\ml{N}_k} \to DMT(k, \mb{Q})$$
for any field $k$, where $DMT(k, \mb{Q})$ is the full rigid tensor subcategory of Veovodsky's triangulated category of geometric motives generated by Tate objects. The precise definition of $DMT(k, \mb{Q})$ can be found in \cite{L}.
\

If we view Adams cdgas as cdgas in the category of representations over $\mb{G}_m$, the above example provides us a motivation to study motives by the general theory of cdgas over a reductive group. Using the theory of cdgas over $GL_2$,  we generalize Spitzweck' equivalence to the case of motives for an elliptic curve without complex multiplication in \cite{Cao}. 
\

We now state our main result together with the outline of this paper. Assume $G$ is a reductive group with an injective central cocharacter and $A$ is cdga over $G$ (Definition \ref{def of cdga over GL_2}). Like the case of Adams graded cdgas, we can also define the derived category of dg $A$-modules, denoted by $\mathcal{D}_{A}^{\G}$ and study their properties, which is the content of Section \ref{Basic definition} - Section \ref{min}. Furthermore if we assume that $A$ is cohomologically connected, the full subcategory $\mathcal{D}_{A}^{\G, f}$ of $\mathcal{D}_{A}^{\G}$ consisting of compact objects works well (Section \ref{tensor structure}). Then we have:
\

\begin{thm}
(Theorem \ref{existence of t-structure} and Theorem \ref{main thm for cdga over $\G$})
Suppose $A$ is cohomologically connected. There exists a non-degenerate $t$-structure on $\mathcal{D}_{A}^{\G, f}$ with heart $\mathcal{H}_{A}^{\G, f}$. Furthermore,
\begin{itemize}
\item
There is a functor:
$$\rho: D^{b}(\mathcal{H}_{A}^{\G, f}) \longrightarrow \mathcal{D}^{\G, f}_{A}.$$
\item
The functor $\rho$ constructed above is an equivalence of triangulated categories if and only if $A$ is $1$-minimal.
\end{itemize}
\end{thm}
The proof of the first part of the above theorem is included in Section \ref{t}. In order to complete the proof of the second part, we need to use the reduced bar construction (Section \ref{bar con}) to give other descriptions about $\ml{H}^f_A$ (Section \ref{heart}). Then we finish the whole proof in Section \ref{main}. If $G = \mb{G}_m$, the above theorem coincides with the works of Kriz and May.  
\

The theory of cdgas over a reductive group is closely related with the weighted completion (see Definition \ref{df of weighted completion}).  We explain their relation in Section \ref{relation}. Another observation is that $\ml{H}^{f}_A$ belongs to a special kind of Tannakian categories (Definition \ref{split tan}), whose Tannakian fundamental groups are a semi-product of a prounipotent algebraic group and a reductive group. In the last section of the paper, we give a precise description about the coordinate ring of the Tannakian fundamental groups of such Tannkain categories by framed objects \cite{BGSV}.

%

\section*{Acknowledgements}
This paper is part of my PhD thesis at Universit\"{a}t Duisburg-Essen  submitted in August, 2016. I want to express my deepest gratitude to my adviser Professor Marc Levine for his constant guidance, encouragement and patience during this work.

\section{Basic definitions} \label{Basic definition}
\begin{con} \label{conv}
Let $G$ be a reductive algebraic group over $\mb{Q}$ and $w : \mb{G}_m \to G$ be a central cocharacter -- that is, the image of $w$ is contained in the center of $G$. We assume that $w$ is nontrivial i.e., injective. Using the map $w$, we can define the weight of representations of $\G$. (See Definition \ref{def of weight}). Fix a finite dimensional faithful representation $\V$ of $G$ with positive weights\footnote{For the existence of such $\V$, we refer to Corollary 2.5 in \cite{DM}.}. 
\end{con}

\begin{df}  \label{def of weight}
Let $V$ be a rational $G$-representation. For any $r \in \mb{Z}$, define the weight $r$ part of $V$ to be a sub representation of $V$:
$$V \lw r \rw =  \{x \in V  |   w(\lambda) \cdot x = \lambda^{r} x \ \ \ \ \ \text{for any} \ \ \  \lambda \in \mb{G}_{m}(k) \}.$$ 
A rational $G$-representation $V$ is called pure of weight $r$ if $V\lw r \rw = V$.
\end{df}
\begin{con}
The Adams degree for a pure weight $r$ representation $W$ of $\G$ over $\mb{Q}$ is defined to be $-r$. Given a complex of  linear $\G$-representations $A^{*}$, the Adams degree $r$ part of $A^{*}$ is denoted by $A^{*} \law r \raw$. We call the category of linear $\G$-representations over $\mb{Q}$ simply as the category of $\G$-representations.
\end{con}
\begin{df} \label{def of cdga over GL_2}
A cdga $(A^{*}, d, \cdot)$ over $\G$ consists of a complex $(A^{*}, d)$ in the category of $\G$-representations, where $d = \oplus_{n} d^{n}: A^{n} \to A^{n+1}$ is a homomorphism between $\G$-representations, satisfying:
\begin{itemize}
\item
there exists a homomorphism of complexes of $\G$-representations: $\cdot: A^{*} \otimes A^{*} \rightarrow A^{*}$, which is unital, graded commutative and associative.
\item
$d^{n+m}(a \cdot b) = d^{n}a \cdot b + (-1)^{n} a \cdot d^{m}b$, where $a \in A^{n}, b \in A^{m}$.
\item
the Adams grading gives a decomposition of $A^{*}$ into subcomplexes $A^{*} = \oplus_{r \in \mb{Z}}A^{*}\law r \raw$ and $\mb{Q}$ (the trivial $\G$-representation) is a direct summand of $A^{*} \law 0 \raw$. 
\end{itemize}
$A^{*}$ is called Adams connected if the Adams decomposition satisfies $A^{*} = \oplus_{r \geq 0}A^{*}\law r \raw$ and $A^{*} \law 0 \raw = \mb{Q}$.
Furthermore, $A^{*}$ is called connected (resp. cohomologically connected) if $A^{n} = 0$ for $n < 0$ and $A^{0} = \mathbb{Q}$ (resp. $H^{n}(A^{*}) = 0$ for $n < 0$ and $H^{0}(A^{*}) = \mathbb{Q}$).
\

For $x \in A^{n}\law r \raw$, we call $n$ the cohomological degree of $x$, denoted by $n = deg(x)$, and $r$ the Adams degree of $x$, denoted by $r = |x|$.
\end{df}

\begin{df} \label{def of dg A-module}
Let $A$ be a cdga over $\G$. A dg $A$-module $(M^{*}, d)$ over $\G$ consists of a complex $M^{*}$ of $\G$-representations with the differential $d$, together with a map $A^{*} \otimes M^{*} \rightarrow M^{*}, a \otimes m \rightarrow a \cdot m$, which makes $M^{*}$ into a $A^{*}$-module, and satisfies the Leibniz rule
$$d(a \cdot m) = da \cdot m + (-1)^{deg a}a \cdot dm;  a \in A^{*}, m \in M^{*}.$$
\end{df}

\begin{rmk}
By definition, there exists a decomposition of $M^{*}$ into subcomplexes $M^{*} = \oplus_{s}M^{*}\law s \raw$ satisfying $A^{*}\law r \raw \cdot M^{*} \law s \raw \subset M^{*} \law r+s \raw$, which is called the Adams decomposition.
\end{rmk}
\begin{df}
Let $M$ and $N$ be two dg $A$-modules. A morphism $f$ between $M$ and $N$ is a morphism between the underlying complexes of $\G$-representations of $M$ and $N$ such that $a \cdot f(m) = f(a \cdot m)$ for any $a \in A$ and $m \in M$.
\end{df}
\begin{exe} \label{gen-sphere}
Let $A[n]$ denote the $A^{*}$-module which is $A^{m + n}$ in degree $m$, with  a natural action of $A^{*}$ by multiplication. Given $A^{*}$ a cdga over $\G$, we let $A\langle r \rangle[n]$ be $A^{*}$-module which is $\bigoplus_{t \in \mb{Z}}A^{m + n}\law t \raw \otimes \V^{\otimes r} \law s-t \raw$ in bi-degree $(m, s)$, with the action given by multiplication. More generally,
given any $\G$-representation $W$, $A[n] \otimes W$, with $\bigoplus_{t \in \mb{Z}}A^{m + n}\law t \raw \otimes W \law s-t \raw$ in degree $(m, s)$, is also a dg $A$-module over $\G$. When $W$ is a rational representation of $\G$, $A[n] \otimes W$ is called the generalized sphere $A$-modules for any $n \in \mb{Z}$.
\end{exe}
\begin{df} \label{df of cell modules}
A dg $A$-module $M$ is a cell module if
\begin{enumerate}
\item \label{cell mod-1}
There is an isomorphism of $A$-modules in the category of $\G$-representations:
$$\oplus_{j \in J}A [-n_{j}] \otimes V_{j} \rightarrow M,$$
where all the $V_j$ are rational representations of $\G$ and all $n_j$ are integers.
\item \label{cell mod-2}
There is a filtration on the index set $J$:
$$J_{-1} = \emptyset \subset J_{0} \subset J_{1} \cdots \subset J$$
such that $J = \bigcup_{n=0}^{\infty}J_{n}$ and for $j \in J_{n}$,
$$db_{j} = \sum_{i \in J_{n-1}}a_{ij}b_{i},$$
where $b_j$ is in the cohomological degree $n_j$ part of the complex $Hom_{\G}(V_j, M)$ and $a_{ij}$ is in $A$.
\end{enumerate}
A finite cell module is a cell module with finite index set $J$.
\end{df}
\begin{rmk}
Given $M$ a cell module, using the condition \ref{cell mod-1} and \ref{cell mod-2}, one can construct a filtration of sub cell modules $M_n$, where $M_n$ is isomorphic to $\oplus_{j \in J_n}A [-n_{j}] \otimes V_{j}$ as complexes of $\G$-representations\footnote{In other words, a cell $A$-module $M$ can be considered as the union of an expanding sequence of sub $A$-modules $M_{n}$ such that $M_{0} = 0$ and $M_{n + 1}$ is the cofiber of a map $\phi_{n}: F_{n} \rightarrow M_{n}$, where all the $F_{n}$ are generalized sphere $A$-modules.}. $\{M_n\}_{n \in \mb{Z}_{\geq 0}}$ is called the sequential filtration of $M$. 
\end{rmk}
\begin{df} \label{df of cell module of Tate type}
Assume that the dimension of the fixed faithful representation $\V$ is $n$ and denote $\wedge^n \V$ by $det$. A cell module is called Tate-type if all the generalized sphere modules appearing in the first condition of Definition \ref{df of cell modules} are of the form $A[-n] \otimes det^{\otimes r}$ for some $r, n \in \mb{Z}$. 
\end{df}
\begin{df}
A rational representation $V$ of $\G$ has non-positive Adams degrees if $V \law r \raw = 0$ for any $r > 0$.
\end{df}
\begin{df}
A cell module is called effective if all the generalized sphere modules appearing in the  definition are of the form $A[-n] \otimes V_j$  with $V_j$ a direct summand of $\V^{\otimes i_j}$ for some $i_j \in \mb{Z}_{\geq 0}$ . 
\end{df}
We denote the category of dg $A$-modules over $\G$ by $\mathcal{M}_{A}^{\G}$, the category of cell $A$-modules by $\mathcal{CM}_{A}^{\G}$ and the category of finite cell modules by $\mathcal{CM}_{A}^{\G, f}$. Furthermore, we denote the category of effective cell $A$-modules by $\mathcal{CM}^{\G, eff}_{A}$ and the category of finite effective cell modules by $\mathcal{CM}_{A}^{\G, eff, f}$.  Finally we denote the full subcategory of cell $A$-modules of Tate-type by $\ml{CM}_{A}^{\mb{G}_m}$.

\section{The derived category of dg modules}
\label{construction for the derived category}
Let $A$ be a cdga over $\G$ and let $M$ and $N$ be dg $A$-modules. Let $\mathcal{H}om_{A}(M, N)$ be the dg $A$-module over $\G$ with $\mathcal{H}om_{A}(M, N)^{n}$ consisting of linear maps $f: M \rightarrow N$ with $f(M^a) \subset N^{a+n}$ and with the differential $d$ defined by $df(m) = d(f(m)) - (-1)^{n}f(dm)$ for $f \in \mathcal{H}om_{A}(M, N)^{n}$.  Let $Hom_{A}(M, N)$ be the dg $\mb{Q}$-module over $\G$ with $Hom_{A}(M, N)^{n}$ consisting of linear maps $f: M \rightarrow N$ with $f(M^a) \subset N^{a+n}$, $ f(am) = (-1)^{np}af(m)$
for $a \in A^{p}$ and $m \in M^a$, and with the differential $d$ defined by $df(m) = d(f(m)) - (-1)^{n}f(dm)$ for $f \in Hom_{A}(M, N)^{n}$. 
\begin{df}
For $f: M \rightarrow N$ a morphism of  dg $A$-modules, we let $Cone(f)$ be the dg $A$-module with:
$$Cone(f)^{n}(r) = N^{n}(r) \oplus M^{n+1}(r)$$
and the differential is given by $d(n, m) = (dn + f(m), -dm)$.
\end{df}
Given $M$ a dg $A$-module, we let $M[1]$ denote a dg $A$-module such that $M[1]^{n} = M^{n+1}$ with the differential $-d$, where $d$ is the differential of $M$. Then we have the following sequence:
$$M \xrightarrow{f} N \xrightarrow{i} Cone(f) \rightarrow M[1],$$
which is called a cone sequence.
\begin{df} \label{def-hom}
We let $\mathcal{K}_{A}^{\G}$ denote the homotopy category of the category of dg $A$-modules over $\G$. The objects are the same as $\mathcal{M}_{A}^{\G}$ and $$Hom_{\mathcal{K}_{A}^{\G}}(M, N) = Hom_{\G}(\mb{Q}, H^{0}(Hom_{A}(M, N))).$$
\

The derived category $\mathcal{D}_{A}^{\G}$ of dg $A$-modules over $\G$ is the localization of $\mathcal{K}_{A}^{\G}$ with respect to quasi-isomorphisms between dg $A$-modules, which are defined as morphisms $M \rightarrow N$ being quasi-isomorphic on the underlying complexes of 
$\mb{Q}$-vector spaces.
\end{df}
Given $M, N \in \mathcal{M}_{A}^{\G}$ (resp.  $\mathcal{CM}_{A}^{\G}$), we can define their direct sum to be the direct sum $M \oplus N$ of the chain complexes of $GL_2$-representations which is equipped with a natural $A$-module structure (resp. cell $A$-module structure). Furthermore, the infinite direct sum exists in both $\mathcal{M}_{A}^{\G}$ and $\mathcal{CM}_{A}^{\G}$.
\begin{lem}
The infinite direct sums defined above  is the categorical sum in $\mathcal{K}_{A}^{\G}$.
\end{lem}
\begin{con}
Let $I$ be the complex
$$\mathbb{Q} \xrightarrow{\delta} \mathbb{Q} \oplus \mathbb{Q}$$
with a free $\mathbb{Q}$ generator $[I]$ in degree $-1$, two free $\mb{Q}$ generators $[0], [1]$ in degree $0$ and $\delta[I] = [0] - [1]$. We have two inclusions $i_0, i_1: \mb{Q} \to I$ sending $1$ to $[0], [1]$ respectively.
\

For $M$ a dg $A$-module, we let $CM = Cone(id_M)$. Notice that the cone $CM$ is the quotient module $M \otimes (I/\mb{Q}[1])$.
\end{con}
Using the same idea of proof as in \cite{KM}, we can show the following theorems.
\begin{thm} \label{HELP}
(HELP) Let $L$ be a cell $A$-submodule of a cell $A$-module $M$. Let $e: N \to P$ be a quasi-isomorphism of dg $A$-modules. Then given maps $f: M \to P, g: L \to N$, and $h: L \otimes I \to P$ such that $f|_{L} = h \circ i_0$ and $e \circ g = h \circ i_1$, there are maps $\widehat{g}, \widehat{h}$ that make the following diagram commute.
\

\begin{center}
\begin{xy}
(210,0)*+{L}="v1";(240,0)*+{L \otimes I}="v2";(270,0)*+{L}="v3";
(225,-15)*+{P}="v4";(255,-15)*+{N}="v5";
(210,-30)*+{M}="v6";(240,-30)*+{M \otimes I}="v7";(270,-30)*+{M}="v8";
{\ar@{->}^{i_0} "v1";"v2"};{\ar@{->}_{i_1} "v3";"v2"};
{\ar@{->} "v1";"v6"};{\ar@{->} "v2";"v7"};{\ar@{->} "v3";"v8"};
{\ar@{->}^{h} "v2";"v4"};{\ar@{->}^{g} "v3";"v5"};
{\ar@{->}_{\ \ e} "v5";"v4"};
{\ar@{->}^{f} "v6";"v4"};{\ar@{->}^{\widehat{h}} "v7";"v4"};{\ar@{->}^{\widehat{g}} "v8";"v5"};
{\ar@{->}^{i_0} "v6";"v7"};{\ar@{->}_{i_1} "v8";"v7"};
\end{xy}
\end{center}
\end{thm}
\begin{proof}
By induction on the filtration on the index set $J$ and pullback along cells not in $L$, we may assume that $M  \cong C(A[n] \otimes W)$ and $L \cong A[n] \otimes W$. Then using the semi-simplicity of the category of $\G$-representations, we can further assume that $W$ is an irreducible $\G$-representation. Let's denote the generator of $W$ by $w^n$.
\

Let $u = w^n \otimes [0]$ and $v = w^{n} \otimes [I]$ be the generators of $C(A[n] \otimes W)$. By definition, we have $d(v) = (-1)^n u$. We also have:
$e \circ g(w^n) = h(w^n \otimes [1])$
and
$f(u) = h(u)$. Therefore
\begin{equation} \nonumber
\begin{split}
&d(h(w^n \otimes [I]) - f(v)) = hd(w^n \otimes [I]) - f(dv) \\
= & h(d(w^n) \otimes [I] + (-1)^n h(w^n \otimes ([0] - [1]))) - (-1)^n f(u)\\
= & (-1)^n h(w^n \otimes [0]) + (-1)^{n+1} h(w^n \otimes [1]) - (-1)^n f(u) \\
= & (-1)^{n+1} h(w^n \otimes [1]) =(-1)^{n+1} e \circ g(w^n).
\end{split}
\end{equation}
Because $e \circ g(w^n)$ is a coboundary and $e$ induces a quasi-isomorphism, we know that $g(w^n)$ is also a coboundary, i.e., there exist $\tilde{n} \in N^{n-1}$ such that $d(\tilde{n}) = g(w^n)$. Then
$p = e(\tilde{n}) + h(w^n \otimes [I]) - f(v)$ is a cocycle. Then using the quasi-isomorphism at $n-1$, there exist a cocycle $n \in N$ and a chain $q \in P$ such that $d(q) = p - e(n)$.
\

We define $\widehat{g}(j) = (-1)^n(\tilde{n} - n)$ and $\widehat{h}(j \otimes [I]) = q$.
\end{proof}
\begin{thm} \label{Whitehead}
(Whitehead) If $M$ is a cell $A$-module and $e: N \rightarrow P$ is a quasi-isomorphism of $A$-modules, then $$e_{*}: Hom_{\mathcal{K}_{A}^{\G}}(M, N) \rightarrow Hom_{\mathcal{K}_{A}^{\G}}(M, P)$$ is an isomorphism. So a quasi-isomorphism between cell $A$-modules is a homotopy equivalence.
\end{thm}
\begin{proof}
The surjectivity is coming from Theorem \ref{HELP}, when we take $L = 0$. The injectivity can be checked when we replace $M$ and $L$ by $M \otimes_{\mb{Q}} I$ and $M \otimes_{\mb{Q}} (\partial I)$ respectively. When $N, P$ are both cell $A$-modules, taking $M = P$, we get a map $f: P \to N$ which corresponds to $id_P$. From the functoriality, $f$ is the homotopy inverse of $e$. 
\end{proof}
\begin{cor}
Let $M, N$ be two dg $A$-modules, and $f: M \to N$ be a morphism between dg $A$-modules. Let $\widehat{M}$ and $\widehat{N}$ be two cell $A$-modules such that $\widehat{M} \xrightarrow{r_M} M$ and $\widehat{N} \xrightarrow{r_N} N$ are quasi-isomphisms. Then there exists a morphism between cell $A$-modules $\widehat{f}: \widehat{M} \to \widehat{N}$ lifting $f$. 
\end{cor}
\begin{proof}
From Theorem \ref{Whitehead}, we know that $$r_{N*}:Hom_{\mathcal{K}_{A}^{\G}}(\widehat{M}, \widehat{N}) \rightarrow Hom_{\mathcal{K}_{A}^{\G}}(\widehat{M}, N)$$
is an isomorphism. Therefore $f \circ r_M \in \mathcal{K}_{A}(\widehat{M}, N)$ have a preimage, which is just $\widehat{f}$.
\end{proof}
\begin{rmk} \label{Hom_K = Hom_D}
From the above proof, we also know that: Given a cell module $M$ and an arbitrary dg $A$-module $N$, we have:
$$Hom_{\mathcal{D}_{A}^{\G}}(M, N) \cong Hom_{\mathcal{K}_{A}^{\G}}(M, \widehat{N}) \cong Hom_{\mathcal{K}_{A}^{\G}}(M, N),$$
where $ \widehat{N}$ is a cell module and $ \widehat{N} \to N$ is a quasi-isomorphism between dg $A$-modules.
\end{rmk}
\begin{thm}  \label{App cell}
(Approximation by cell modules) For any dg $A$-module $M$, there is a cell $A$-module $N$ and a quasi-isomorphism $e: N \rightarrow M$.
\end{thm}
\begin{proof}
We will construct a sequential filtration $N_n$ and compatible maps $e_n: N_n \to M$ inductively. More precisely, we need to construct cell modules $N_n$, whose index set is denoted by $J_n$, satisfy the condition \ref{cell mod-2} in the definition of cell modules. For every pair $(q, r)$, we decompose $H^{q}(M)\law r \raw \cong \oplus_{i}V_{i}$ as the direct sum of  irreducible $\G$-representations $V_{i}$ with the Adams degree $r$. Choosing a splitting of $Ker(M^q \law r \raw \xrightarrow{d} M^{q+1} \law r \raw) \twoheadrightarrow H^{q}(M) \law r \raw$, we think $V_{i}$ as sub $\G$-representations in $M^{q}\law r \raw$, because of the semi-simplicity of the category of $\G$-representations. Then we take $N_{1} = \oplus_{(q, r)} \oplus_{i} A[-q] \otimes V_{i}$ with trivial differential. There is a morphism between dg $A$-modules: $N_{1} \rightarrow M$, which is epimorphism on the cohomologies. Inductively, assume that $e_n: N_n \to M$ has been constructed. Consider the set of the pair of cocycles consisting the pairs of unequal cohomology classes on $N_n$ and mapping under $(e_n)^{*}$ to the same element of $H^{*}(M)$. Choose a pair $W_1^q \law r \raw$ and $W_2^q \law r \raw$ that live in the bidegree $(q,r)$ satisfies above condition, i.e., we can view $W_1^q \law r \raw \oplus W_2^q \law r \raw$ as the kernel of the morphism $e_{n*}$ on the cohomology of bidegree $(q,r)$. (Here one need to take a sign for the second component.) Simply denote $W_1^q \law r \raw \oplus W_2^q \law r \raw$ by $W_1$.
There is a morphism between dg $A$-modules $A[-q] \otimes W_{1}$ to $N_{n}$ extending the map between $\G$-representations $W_1 \to H^q(N)(r)$. Take $N_{n+1}$ to be the pushout of $N_{n}$ and $A[-q] \otimes W_{1} \oplus A[-q] \otimes W_{1}[1]$ over $A[-q] \otimes W_{1}$. Then we have $0 \rightarrow W_{1} \rightarrow H^{q}(N_{n})(r) \rightarrow H^{q}(N_{n+1})(r) \rightarrow 0$. We get $N_{n+1}$ by attaching $N_{n}$ with a generalized sphere dg $A$-module $A[-q] \otimes W_{1}[1]$, which implies $N_{n+1}$ is a cell $A$-module. It is easy to see the differentials on $N_{n+1}$ satisfy the condition \ref{cell mod-2} in the definition of cell modules. Now we have a distinguish triangle of dg $A$-modules:
$$A[-q] \otimes W_1 \xrightarrow{i} N_n \to N_{n+1} \to (A[-q] \otimes W_1)[1].$$
Notice that:
\begin{equation} \nonumber
\begin{split}
& Hom_{A}(A[-q] \otimes W_1, M) \cong Hom_{D(\G)}(W_1, M[q]) \cong Hom_{\G}(W_1, H^q(M)).
\end{split}
\end{equation}
Therefore we have:
\begin{equation} \nonumber
\begin{split}
& Hom_{A}(N_{n+1}, M) \to Hom_A(N_n, M) \\ \xrightarrow{i} & Hom_{A}(A[-q] \otimes W_1, M) \cong Hom_{\G}(W_1, H^q(M)).
\end{split}
\end{equation}
Because $W_{1}$ as a $\G$-representation maps to zero in the cohomology group of $H^{q}(M)\law r \raw$, which implies $i(e_n) = 0$ in $Hom_{\G}(W_1, H^q(M))$, one may find $e_{n+1} \in Hom_{A}(N_{n+1}, M)$, which extends $e_n$. Let $N$ be the direct limit of the $N_{n}$. Then $N$ is a cell module and the morphism $N \rightarrow M$ is a quasi-isomorphism by the construction.
\end{proof}

Putting together with all previous results, we get:
\begin{thm}  \label{derived}
Let $A$ be a cdga over $\G$. Then the functor
$$\mathcal{KCM}_{A}^{\G} \rightarrow \mathcal{D}_{A}^{\G}$$
is an equivalence of triangulated categories.
\end{thm}
\begin{df}
We define $\mathcal{D}^{\G, f}_{A}$ to be the full subcategory of $\mathcal{D}_{A}^{\G}$ whose objects are quasi-isomorphic to some finite cell $A$-module in $\mathcal{D}_{A}^{\G}$.
\end{df}
\begin{rmk}
From the proof of Theorem \ref{derived}, we can also know that $\mathcal{KCM}^{\G, f}_{A} \rightarrow \mathcal{D}^{\G, f}_{A}$ is an equivalence of triangulated categories.
\end{rmk}
\begin{exe} \label{basic example}
Let $A = \mathbb{Q}$, then $\mathcal{KCM}^{\G, f}_{A}$ is just the bounded derived category of the category of rational representations of $\G$, denoted by $D^{b}(\G)$.
\end{exe}
\begin{df}
The full triangulated subcategory of $\mathcal{D}_{A}^{\G}$ generated by effective cell modules is denoted by $\mathcal{D}^{\G, eff} _{A}$. The full triangulated subcategory of $\mathcal{D}_{A}^{\G}$ generated by a family of objects $\{A\langle r \rangle[n] | r \geq 0, n \in \mb{Z}\}$ is denoted by $\ml{T}_{A}^{\G}$.
\end{df}
Recall the definition of the idempotent completion of a dg category. Given $\ml{C}$  a dg category, then its idempotent completion $\ml{C}^{\natural}$ has the objects $(M, p)$ with $p: M \to M$ an idempotent endomorphism in $Z^{0}\ml{C}$ and the hom complex  given by
$$\ml{H}om_{\ml{C}^{\natural}}((M, p), (N, q))^{*} = p^{*}q_{*}\ml{H}om_{\ml{C}}(M, N).$$
\begin{rmk}
In \cite{BS}, Balmer and Schlichting have shown that, for $\ml{A}$ a triangulated category, $\ml{A}^{\natural}$ has a canonical structure of a triangulated category which makes the natural functor $\ml{A} \to \ml{A}^{\natural}$ exact. The same holds for the triangulated tensor categories.
\end{rmk}
\begin{exe}
Let us consider $A = \mathbb{Q}$ and $\G = GL_2$. Then $\mathcal{D}^{GL_2, eff, f}_{A}$ is a subcategory of the bounded derived category of the category of rational representations of $GL_2$ generated by $Sym^{a}(b)$ for $a, b \in \mb{Z}_{\geq 0}$ and denoted by $D^{eff, b}(GL_2)$.
\end{exe}
\section{The weight filtration for dg modules} \label{weight}
In this section, we assume that $A$ is an Adams connected cdga over $\G$. (Definition \ref{def of cdga over GL_2}.)
\begin{df}
A dg $A$-module $M$ is called almost free, if there exists a family of irreducible $\G$-representations $\{ V_{j} \}_{j \in J}$ and morphisms of dg $A$-modules $\phi_{j}: A \otimes V_{j} \rightarrow M$, 
such that the induced morphism:
$$\oplus_{j \in J} A \otimes V_{j} \xrightarrow{\oplus \phi_{j}} M$$
is an isomorphism of graded $A$-modules, which means that, forgetting the differentials, this is an isomorphism between $\G$-representations. We call such $\{ V_{j}, \phi_{j} \}_{j \in J}$ the generating data for $M$.
\end{df}

\begin{exe}
All cell $A$-modules are almost free. Conversely, any cell $A$-module is obtained from the generating data together with suitable differentials.
\end{exe}
The reason for introducing the notion of almost free is that we can define the weight filtration on these data. Assume that $d(\phi_{j}(V_{j})) \subset \oplus_{i \in I}\phi_{i}(A \otimes V_{i})$. Here we restrict $\phi_{j}$ to $A^{*}\law 0 \raw \otimes V_{j} \cong V_{j}$. The left hand side has the Adams degree $|V_{j}|$ and the Adams degree of right hand side is larger than or equal to $|V_{i}|$. So we get $|V_{i}| \leq |V_{j}|$ if $d \phi_{j} \neq 0$. Hence we have the subcomplex
$$W_{n}^{J}M = \oplus_{\{j, |V_{j}| \leq n\}} \phi_{j}(A \otimes V_{j})$$
of $M$.
\begin{rmk}
The subcomplex of $W_{n}^{J}M$ is independent of the choice of the family $\{\phi_{j}\}_{j \in J}$. This is because if we choose another family $\{\phi_{j^{'}}\}$, then the same process as above shows that $\phi_{j^{'}}(V_{j^{'}}) \in W_{n}^{J}M$ and hence $W_{n}^{J^{'}}M \subset W_{n}^{J}M$. By symmetry, we get the result. So we delete the $J$ in the definition.
\end{rmk}
This gives us the increasing filtration as a dg $A$-module
$$W_{*}M: \cdots \subset W_{n}M \subset W_{n+1}M \subset \cdots \subset M$$
with $M = \cup_{n} W_{n}M$.
\

In the same way, we can define $W_{n/n^{'}}M$ as the cokernel of the inclusion $W_{n^{'}}M \rightarrow W_{n}M$ for $n \geq n^{'}$. Write $gr_{n}^{W}$ for $W_{n / n-1}$ and $W^{> n}$ for $W_{\infty / n}$.
\

$W_{n}$ defines an endofunctor in $\mathcal{CM}_{A}^{\G}$. Furthermore, $\{W_n\}_{n \in \mb{Z}}$ form a functorial tower of endofunctors on $\mathcal{KCM}_{A}^{\G}$:
$$\cdots \rightarrow W_{n} \rightarrow W_{n+1} \rightarrow \cdots \rightarrow id.$$
\begin{rmk}
\begin{itemize}
\item
The endofunctor $W_{n}$ is exact for all $n$.
\item
For $m \leq n \leq \infty$, the sequence of endofunctors $W_{m} \rightarrow W_{n} \rightarrow W_{n/m}$ can extend to a distinguish triangle of endofunctors, i.e., for any $M \in \mathcal{KCM}_{A}^{\G}$, we have a distinguish triangle $W_{m}M \rightarrow W_{n}M \rightarrow W_{n/m}M \rightarrow$ in $\mathcal{KCM}_{A}^{\G}$.
\end{itemize}
\end{rmk}

\begin{rmk}
Using the isomorphism of categories between $\mathcal{KCM}_{A}^{\G}$ and $\mathcal{D}_{A}^{\G}$, we could define the tower of exact endofunctors on $\mathcal{D}_{A}^{\G}$
$$\cdots \rightarrow W_{n} \rightarrow W_{n+1} \rightarrow \cdots \rightarrow id.$$
Similarly we define $W_{n/n^{'}}$, $gr_{n}^{W}$ and $W^{> n}$ on $\mathcal{D}_{A}^{\G}$.
\end{rmk}
The existence of the weight filtration provides a powerful tool for showing a lot of properties of dg $A$-modules.

\section{Tensor structure}  \label{tensor structure}
Recall that the Hom functor $\mathcal{H}om_{A}(M, N)$ defines a bi-exact bi-functor:
$$\mathcal{H}om_{A}: (\mathcal{KCM}_{A}^{\G})^{op} \otimes \mathcal{KCM}_{A}^{\G} \rightarrow \mathcal{D}_{A}^{\G},$$
which gives a well-defined derived functor of  $\mathcal{H}om_{A}$ between the derived categories of dg $A$-modules (also the derived categories of finite cell modules) by Proposition \ref{derived}:
$$R\mathcal{H}om_{A}: (\mathcal{D}_{A}^{\G})^{op} \otimes \mathcal{D}_{A}^{\G} \rightarrow \mathcal{D}_{A}^{\G}.$$
In this section, we use these constructions to define the tensor structure on $\mathcal{D}_{A}^{\G}$.
\

Firstly, we define a tensor structure on $\ml{T}_{A}^{\G}$. It is defined on the generator $\{A\langle r \rangle[n]\}$ by:
$$A\langle r \rangle[n] \otimes_{A} A\langle s \rangle[m] = A\langle r + s \rangle [m+n].$$
Using the approximation theorem, we get a derived tensor product $\otimes^{\mb{L}}$ on $\ml{T}_{A}^{\G}$. Then it will induce a tensor structure on $(\ml{T}_{A}^{\G})^{\natural}$. More precisely, take $(M, p), (N, q) \in (\ml{T}_{A}^{\G})^{\natural}$, then
$$(M, p) \otimes^{\mb{L}}_{A} (N, q) = (M \otimes^{\mb{L}}_{A} N, p \otimes q).$$
From above construction, we have a functor:
$$\otimes_{A}^{\mb{L}}: (\ml{T}_{A}^{\G})^{\natural} \otimes (\ml{T}_{A}^{\G})^{\natural} \to (\ml{T}_{A}^{\G})^{\natural}.$$
Next we want to extend the above tensor structure on $(\ml{T}_{A}^{\G})^{\natural}$ to $\ml{D}^{\G}_{A}$.
\

Notice that given a generalized sphere modules $A \otimes W$, where $W$ is a rational $\G$-representation, there exists a positive integer  $n$ big enough, such that $A \otimes W(n)$ is in $(\ml{T}_{A}^{\G})^{\natural}$. Here $(1)$ means the $1$ dimensional $G$- representation $\wedge^n \V$ and $n = \dim \V$.   
\begin{df} \label{df of tensor}
Given two generalized sphere modules $A \otimes W_i, i = 1,2$, and $n_i \in \mb{Z}_{\geq 0}$, such that $A \otimes W_i(n_i)$ is effective, then we define:
\begin{equation} \nonumber
(A \otimes W_1) \otimes^{\mb{L}}_{A} (A \otimes W_2) = \mathcal{H}om_{A}(A(n_1+n_2), (A \otimes W_1(n_1)) \otimes^{\mb{L}}_{A} (A \otimes W_2(n_2)) ).
\end{equation}
By the definition of the internal hom, it's easy to see that the definition is independent of the choice of $n_i$.
\end{df}
By Theorem \ref{App cell}, we get a well-defined derived functor of  $\otimes_{A}$:
$$\otimes_{A}^{\mathbb{L}}: \mathcal{D}_{A}^{\G} \otimes \mathcal{D}_{A}^{\G} \rightarrow \mathcal{D}_{A}^{\G}.$$
\begin{rmk}
Two formal properties are listed below.
\begin{itemize}
\item
These bi-functors are adjoint, i.e.,
$$R\mathcal{H}om_{A}(M \otimes_{A}^{\mathbb{L}} N, K) \cong R\mathcal{H}om_{A}(M , R\mathcal{H}om_{A}(N, K)).$$
\item
The derived tensor product makes $\mathcal{D}_{A}^{\G}$ into a triangulated tensor category with unit $A$ and $\mathcal{D}_{A}^{\G, f}$ are triangulated tensor subcategories.
\end{itemize}
These properties allow us to apply the category duality theory in \cite{LMS}.
\end{rmk}
\begin{con}
Denote $M^{\vee} = R\mathcal{H}om_{A}(M, A)$.
\end{con}
\begin{df} \label{rigid}
An object $M \in \mathcal{D}_{A}^{\G}$ is called rigid, if there exists an $N \in \mathcal{D}_{A}^{\G}$ and morphisms $\delta: A \rightarrow M \otimes_{A}^{\mathbb{L}} N$ and $\epsilon: N \otimes_{A}^{\mathbb{L}} M \rightarrow A$ such that:
$$(id_{M} \otimes \epsilon) \circ (\delta \otimes id_{M}) = id_{M}$$
$$(id_{N} \otimes \delta) \circ (\epsilon \otimes id_{N}) = id_{N}$$
\end{df}
\begin{df} \label{finite obj}
An object $M \in \mathcal{D}_{A}^{\G}$ is finite if there exists a coevaluation map $\tilde{\eta}: A \rightarrow M \otimes^{\mathbb{L}} M^{\vee}$ such that the diagram
\begin{center}
\begin{xy}
(220,0)*+{A}="v1";(260,0)*+{M \otimes^{\mathbb{L}} M^{\vee}}="v2";
(220,-20)*+{R\mathcal{H}om_{A}(M, M)}="v3";(260,-20)*+{M^{\vee} \otimes^{\mathbb{L}} M}="v4";
{\ar@{->}^{\tilde{\eta}} "v1";"v2"};{\ar@{->}^{\gamma} "v2";"v4"};
{\ar@{->}_{\eta} "v1";"v3"};{\ar@{->}_{\mu} "v4";"v3"};
\end{xy}
\end{center}

\noindent commutes. Here $\eta$ and $\mu$ are given by the adjunction. $\gamma$ changes the places of these two modules.
\end{df}
\begin{rmk}
By Theorem 1.6 of \cite{LMS}, $M$ is rigid if and only if the function
$$\epsilon_{*}: Hom_{\mathcal{D}_{A}^{\G}}(W, Z \otimes^{\mathbb{L}}_{A} N) \rightarrow Hom_{\mathcal{D}_{A}^{\G}}(W \otimes^{\mathbb{L}}_{A} M, Z)$$
is a bijection for all $W$ and $Z$, where $\epsilon_{*}(f)$ is the composite
$$W \otimes^{\mathbb{L}}_{A} M \xrightarrow{f \otimes 1} Z \otimes^{\mathbb{L}}_{A} N \otimes^{\mathbb{L}}_{A} M \xrightarrow{1 \otimes \epsilon} Z \otimes^{\mathbb{L}}_{A} A \cong Z.$$
These conditions are also equivalent to saying that $M$ is finite.
\end{rmk}
In the following, we will discuss the relations between finite objects in $\mathcal{D}_{A}^{\G}$ and finite cell modules.
\begin{df}
We say that a cell module $N$ is a summand of a cell module $M$ in $\mathcal{D}_{A}^{\G}$ if there is a homotopy equivalence of $A$-modules between $M$ and $N \oplus N^{'}$ for some cell $A$-module $N^{'}$.
\end{df}
Following the same proof as Theorem 5.7 in Part III of \cite{KM}, we can get:
\begin{lem}    
A cell module $M$ is rigid if and only if it is a summand of a finite cell module in $\mathcal{D}_{A}^{\G}$.
\end{lem}

\begin{rmk} \label{FCM}
Let $\ml{F}\mathcal{CM}_{A}^{\G}$ be the full subcategory of $\mathcal{CM}_{A}^{\G}$ whose objects are the direct summands up to homotopy of finite cell $A$-modules. Then the homotopy category $\ml{KFCM}_{A}^{\G}$ is the idempotent completion of $\ml{D}^{\G, f}_A$. The above lemma implies that $\ml{KFCM}^{\G}_A$ is the largest rigid tensor subcategory of the derived category $\ml{D}^{\G}_A$. See section 5 in Part III of \cite{KM}. In particular, $\ml{D}^{\G, f}_A$ is a rigid tensor subcategory of $\ml{KFCM}^{\G}_A$. 
\end{rmk}
\begin{thm} \label{criterion for rigid}
Let $A$ be an Adams connected cdga over $\G$. Then $M \in \mathcal{D}_{A}^{\G}$ is rigid if and only if $M \in \mathcal{D}_{A}^{\G, f}$, which implies that there is an equivalence between $\ml{D}_A^{\G, f}$ and $\ml{KFCM}^{\G}_A$.
\end{thm}
\noindent $Proof.$ It depends on the following lemma.
\begin{lem} \label{closed under summand}
Assume that $A$ is an Adams connected cdga over $\G$. Let $M$ be a finite cell $A$-module. Suppose $N$ is a summand of $M$ in $\mathcal{D}_{A}^{\G}$. Then there is a finite $A$-cell module $M^{'}$ with $N \cong M^{'}$ in $\mathcal{D}_{A}^{\G}$.
\end{lem}
\begin{proof} 
By Theorem \ref{App cell}, we can assume that $N$ is a cell module. By our assumption, we have $M = N \oplus N^{'}$ in $\mathcal{KCM}_{A}^{\G}$. Since $M$ is finite, there is a minimal $n$ such that $W_{n}M \neq 0$. Thus $W_{n-1}N$ is homotopy equivalent to zero. We may assume that $W_{n-1}N = 0$ in $\mathcal{CM}_{A}^{\G}$. Similarly, we may assume that $M = W_{n+r}M$ and $N = W_{n+r}N$ in $\mathcal{CM}_{A}^{\G}$ for some $r \geq 0$. Then we proceed by induction on $r$.
\

Choose generating data $\{ V_{j}, \phi_{j} \}_{j \in J}$ for $W_{n}M$. Let us prove that $W_{n}M = A \otimes V$ for a finite complex of $\G$-representations $V$. Indeed, by the definition of the weight functor and $W_{n-1}M = 0$, we can get an isomorphism:
$$W_{n}M = \oplus_{|V_{j}| = n} \phi_{j}(A \otimes V_{j}).$$
Notice that $d(\phi_{j}(V_{j})) \subset \oplus_{i}\phi_{i}(A \otimes V_{i})$ and  all these $|V_{i}|$'s have the same value. Using $A^{*}\law 0 \raw = \mathbb{Q}$, we get $d(\phi_{j}(V_{j})) \subset \oplus \phi_{i}(V_{i})$. Set $V = \oplus_{j \in J} \phi_{j}(V_{j})$, which is a complex of $\G$-representations. So we have $W_{n}M = A \otimes V$ as dg $A$-modules. Because the category of $\G$-representations is semisimple, we can assume that all differentials of $V$ are zero. 
\

Let $p: M \rightarrow M$ be the composition of the projection $M \rightarrow N$ and the inclusion $N \rightarrow M$. Then we can see $W_{n}p = id \otimes q$, where $q: V \rightarrow V$ is an idempotent of $V$. $V$ is a direct sum of $\G$-representations with some shifts. Thus $W_{n}N \cong A \otimes im(q)$. We finish the case of $r = 0$.
\

Using the distinguished triangle
$$W_{n}N \rightarrow N \rightarrow W_{n+r/n}N \rightarrow W_{n}N[1],$$
we can replace $N$ with the shifted cone of the map $W_{n+r/n}N \rightarrow A \otimes im(q)[1]$. Since $W_{n+r/n}N$ is a summand of $W_{n+r/n}M$, by induction, we get that $W_{n+r/n}$ is homotopy equivalent to a finite cell module. So the cone of $W_{n+r/n}N \rightarrow A \otimes im(q)$ is also homotopy equivalent to a finite cell module.
\end{proof}
\begin{cor}
Assume $A$ is an Adams connected cdga over $\G$. Then $\ml{D}_A^{\G, f}$ is idempotent complete.
\end{cor}

\section{Base change}
\begin{lem} \label{cell preserve q-i}  
Let $N$ be a cell module. Then the functor $M \otimes_{A} N$ preserves exact sequences and quasi-isomorphisms in the variable $M$.
\end{lem}
\begin{proof}
Because $N$ is a cell module, $N$ has a sequential filtration $\{N_{n}\}$, and $N_{n+1}$ is given by the extension of $N_{n}$ by the generalized sphere modules. Notice that  Lemma \ref{cell preserve q-i} is true for generalized sphere modules. By the induction of the filtration and passage to the colimits, we get the result for the general case.
\end{proof}
If $\phi: A \rightarrow B$ is a homomorphism of cdgas over $\G$, we have the functor
$$ \otimes_{A} B: \mathcal{M}_{A}^{\G} \rightarrow \mathcal{M}_{B}^{\G}$$
which induces a functor on cell modules and the homotopy category
$$\phi_{*}: \mathcal{KCM}_{A}^{\G} \rightarrow \mathcal{KCM}_{B}^{\G}.$$
So we have a base change functor on the derived categories level,
$$\phi_{*}: \mathcal{D}_{A}^{\G} \rightarrow \mathcal{D}_{B}^{\G}.$$
\begin{rmk}
The restriction of $\phi_{*}$ on finite objects gives the functor on the bounded case.
\end{rmk}
\begin{prop} \label{preserve under qi}
If $\phi$ is a quasi-isomorphism, then $\phi_{*}$ is an equivalence of tensor triangulated categories.
\end{prop}
\begin{proof} Firstly, there is an isomorphism:
$$Hom_{\mathcal{M}_{B}^{\G}}(B \otimes_{A} M, N) \cong Hom_{\mathcal{M}_{A}^{\G}}(M, \phi^{*}N),$$
for $M \in \mathcal{M}_{A}^{\G}$ and $N \in \mathcal{M}_{B}^{\G}$. Here $\phi^{*}$ is the pullback functor, which means that, for a given dg $B$-module, there is a natural dg $A$-module structure.
\

Then we have:
$$Hom_{\mathcal{K}_{B}^{\G}}(B \otimes_{A} M, N) \cong Hom_{\mathcal{K}_{A}^{\G}}(M, \phi^{*}N).$$
Using Remark \ref{Hom_K = Hom_D}, we get:
\begin{equation} \nonumber
\begin{split}
& Hom_{\mathcal{D}_{B}^{\G}}(B \otimes_{A} M, N) \cong  Hom_{\mathcal{K}_{B}^{\G}}(B \otimes_{A}  \widehat{M}, N) \\
\cong & Hom_{\mathcal{K}_{A}^{\G}}(B \otimes_{A}  \widehat{M}, \phi^{*}N) \cong Hom_{\mathcal{D}_{A}^{\G}}(B \otimes_{A}  M, \phi^{*}N),
\end{split}
\end{equation}
where $\widehat{M}$ is a cell $A$-module quasi-isomorphic to $M$.
\

Next, we will check that the unit of the adjunction and the counit are both quasi-isomorphisms.
\

For the unit of the adjunction, if $M$ is a cell dg $A$-module, then
$$\phi \otimes Id: M \cong A \otimes_{A} M \to \phi^{*}(B \otimes_{A} M)$$
is a quasi-isomorphism of $A$-modules. Firstly, assume that $M = B$. By assumption, we know that $\phi^{*}B$ is quasi-isomorphic to $A$ as a dg $A$-module. Then assume $M = A[n] \otimes W$ for $W$ a $\G$-representation.  $\phi^{*}(B \otimes_{A} A[n] \otimes W)$ is the same as $\phi^{*}(B[n] \otimes W)$. The latter is naturally quasi-isomorphic to $A[n] \otimes W$. If $M$ is cell module, using the induction on the length of its sequential filtration, we can get the desired quasi-isomorphism.
\

For the counit part, given $N$ a dg $B$-module, and choosing a quasi-isomorphim of dg $B$-module $\widehat{N} \to N$, where $\widehat{N}$ is cell $B$-module, then we have:
$$B \otimes_{A} \widehat{N} \to B \otimes_{B} N \cong N,$$
which is also a quasi-isomorphism.
\end{proof}

\begin{cor}
Assume that $A$ and $B$ are Adams connected cdgas over $\G$. If $\phi$ is a quasi-isomorphism, then $$\phi_{*}: \mathcal{D}_{A}^{ \G, f} \rightarrow \mathcal{D}_{B}^{\G, f}$$
is an equivalence of triangulated tensor categories.
\end{cor}
\begin{proof}
Notice that an equivalence between tensor triangulated categories induces an equivalence on the subcategories of rigid objects. By Proposition \ref{preserve under qi}, we know that $\mathcal{D}_{A}^{\G}$ and  $\mathcal{D}_{B}^{\G}$ are equivalent. Then by Theorem \ref{criterion for rigid}, we know that $\phi$ induces an equivalence between $\mathcal{D}_{A}^{\G, f}$ and $\mathcal{D}_{B}^{\G, f}$.
\end{proof}

\begin{rmk} \label{imp trick}
For any cdga $A$ over $\G$, we have a morphism $\delta: \mb{Q} \to A$, which sends $A^*\law 0 \raw$ to $A$.
Then, for any $M \in \mathcal{M}_{\mb{Q}}^{\G}$ and $N \in \mathcal{M}_{A}^{\G}$, we have:
$$Hom_{\mathcal{D}_{A}^{\G}}(A \otimes M, N) \cong Hom_{\mathcal{D}_{\mb{Q}}^{\G}}(M, \delta^{*}N).$$
Here $\delta^{*}$ is the forgetful functor, which forgets the $A$-module structure. We will omit $\delta^{*}$ in the computation later.
\end{rmk}

\section{Minimal models} \label{min}
In the rest of this chapter, we always assume that the cdgas are Adams connected.
\begin{df}
A cdga $A$ over $\G$ is said to be generalized nilpotent if:
\begin{itemize}
\item
$A$ is a free commutative graded algebra over $\G$, i.e, $A = Sym^{*}E$ for some $\mathbb{Z}_{> 0}$-graded $\G$-representations $E$. (Or a complex of $\G$-representations concentrated in position degrees and with zero differentials). 
\item
For $n \geq 0$, let $A \langle n \rangle \subset A$ be the subalgebra generated by the elements of degree $\leq n$. Set $A \langle n+1, 0 \rangle = A \langle n \rangle$ and for $q \geq 0$ define $A \langle n+1, q+1 \rangle$ inductively as the subalgebra generated by $A \langle n \rangle$ and $$A \langle n+1, q+1 \rangle^{n+1} = \{ x \in A \langle n+1 \rangle | dx \in A \langle n+1, q \rangle \}.$$ Then for all $n \geq 0$, $A \langle n+1 \rangle = \cup_{q \geq 0}A \langle n+1,q \rangle.$
\end{itemize}
A cdga $A$ over $\G$ is called nilpotent, if for each $n \geq 1$, there is a $q_{n} \in \mb{Z}_{\geq 0}$ such that $A \langle n \rangle = A \langle n, q_{n} \rangle$ in the second condition above.
\end{df}

\begin{df}
A connected cdga $A$ over $\G$ is minimal if it is a free commutative graded algebra over $\G$ with decomposable differential: $d(A) \subset {(IA)}^{2}$. $IA$ is the fundamental ideal, i.e., $IA = Ker(A \rightarrow \mathbb{Q} \cong A^{0}\law 0 \raw)$.
\end{df}
\begin{con} 
For a cdga $A$ over $\G$, we let $QA$ be $IA/ (IA \cdot IA)$.
\end{con}
\begin{prop} \label{generalized nil equi minimal}
If a connected cdga $A$ over $\G$ is generalized nilpotent, then it is minimal. Conversely, if $A$ is a minimal connected cdga over $\G$ and $A^{q} \law r \raw = 0$ unless $2r \geq q$, then $A$ is generalized nilpotent.
\end{prop}

\begin{proof} The proof is the same as Proposition 2.3 in Part of IV of \cite{KM}. If $A$ is generalized nilpotent, then $d(A) \subset (IA)^{2}$ is the consequence of the double induction on $n$ and $q$. Assume $A$ is minimal. Suppose that $A$ is not generalized nilpotent and let $n$ be minimal such that there is an element $a$ which does not belong to any $A\langle n, q \rangle$. Assume that $a$ is also the element which has the minimal Adams degree. Consider any summand $bc$ of the decomposable element $da $. One can assume that $0 < deg(b) \leq deg(c)$. Then $bc \in A\langle n-1 \rangle$ or $deg(b) = 1$. Since $A^{q} \law r \raw = 0$ unless $2r \geq q$, $b$ and $c$ have strictly lower Adams grading than $a$. So $b, c$ are in some $A\langle n, q \rangle$. Therefore $da$ is in some $A\langle n, q \rangle$, so is $a$.
\end{proof} 

\begin{df} \label{def of 1-min}
Let $A$ be a cdga over $\G$. Given a positive integer $n$, an $n$-minimal model\footnote{$n$ is also allowed to be $\infty$. In this case, we mean that there is a map of cdgas over $G$: $A\{\infty\} \xrightarrow{s} A$ with $A\{\infty\}$ generalized nilpotent and $s$ a quasi-isomorphism.} of $A$ over $\G$ is a map of cdgas over $\G$:
$$s: A\{n\} \longrightarrow A,$$
 with $A\{n\}$ generalized nilpotent and generated as an algebra in degrees $\leq n$, such that $s$ induces an isomorphism on $H^{m}$ for $1 \leq m \leq n$ and an injection on $H^{n+1}$.
\end{df}
\begin{prop}   
Let $A$ be a cohomologically connected cdga $A$ over $\G$. Then for each $n = 1, 2, \cdots, \infty$, there is an $n$-minimal model $A\{n\}$ over $\G$: $A\{n\} \rightarrow A$.
\end{prop}
\begin{proof} Follow the idea of Proposition 2.4.9 in \cite{L}. Because $A$ is cohomologically connected, we have a canonical decomposition $A = \mathbb{Q} \oplus IA$. Let $E_{10}(1) \subset I^{1}\law 1 \raw$ be the $\G$-representation $H^{1}(I)\law 1 \raw$, and we need to think it as a sub-module of $I^{1}\law 1 \raw$. We give it cohomological degree $1$ and Adams degree $1$. Then we have a natural inclusion $E_{10}(1) \rightarrow A$, which extends to $Sym^{*}E_{10}(1) \rightarrow A$ using the algebra structure of $A$. In fact, this is a map between cdgas over $\G$ and induces an isomorphism on $H^{1}(-)\law 1 \raw$.
\

Then one can adjoint elements in cohomological degree $1$ and Adams degree $1$ to kill elements in the kernel of the map on $H^{2}(-)\law 1 \raw$. 
So we have a $\mb{Z}$-graded $\G$-representations $E_{1}(1)$, of Adams degree $1$ and cohomological degree $1$, a generalized nilpotent cdga $A_{1,1} = Sym^{*}E_{1}(1)$ over $\G$ and a map of cdgas over $\G$: $A_{1, 1} \rightarrow A,$ which induces an isomorphism on $H^{1}(-)\law 1 \raw$ and an injection on $H^{2}(-)\law 1 \raw$.
\

We also have a canonical decomposition of $A_{1, 1} = \mathbb{Q} \oplus I_{1, 1}$. 
\

Notice that $H^{p}(I_{1, 1}\law r \raw)= 0$ for $r > 1, p \leq 1$. This is because that the lowest degree of cohomology of $I_{1,1}\law r \raw$ is coming from $Sym^{r}E_{1}(1)$ and all the elements of $E_{1}(1)$ have cohomological degree $1$.
Iterating this process, one can construct the Adams degree $\leq 1$ part of the $n$-minimal model in case $n > 1$. This gives us a generalized nilpotent cdga over $\G$,
$$A_{1, n} = Sym^{*}E_{n}(1),$$
with $E_{n}(1)$ in Adams degree $1$ and cohomological degrees $1, 2, \cdots, n$ together with a map over $\G$: $A_{1, n} \rightarrow A$, which induces an isomorphism on $H^{i}(-)\law 1 \raw$ for $1 \leq i \leq n$ and an injection for $i = n + 1$. In addition, letting $A_{1, n} = \mathbb{Q} \oplus I_{1, n}$, we have $H^{p}(I_{1, n}\law r \raw) = 0$ for $r > 1, p \leq 1$.
\

Suppose we have constructed $\mb{Z}$-graded $\G$-representations:
$$E_{n}(1) \subset E_{n}(2) \subset \cdots \subset E_{n}(m)$$
where $E_{n}(j)$ have Adams degrees $1,\cdots,j$ and cohomological degrees $1,\cdots,n$, a differential on $A_{n,m} = Sym^{*}E_{n}(m) $ making $A_{m, n}$ a generalized nilpotent cdga over $\G$, and a map $A_{m, n} \rightarrow A$ of cdgas over $\G$ that is an isomorphism on $H^{i}(-)\law j \raw$ for $1 \leq i \leq n, j \leq m$, and an injection for $i = n + 1, j \leq m$. 
\

If $A_{m, n} = \mathbb{Q} \oplus I_{m, n}$, then $H^{p}(I_{m, n} \law r \raw) = 0$ for $r > m, p \leq 1$. Extending $E_{n}(m)$ to $E_{n}(m + 1)$ by repeating the construction for $E_{n}(1)$ above. Then one can check the above condition sill hold. The induction goes through.
\

Taking $E_{n} = \cup_{m}E_{n}(m)$, we have a differential on $A\{n\} = Sym^{*}E_{n}$ making $A\{n\}$ a generalized nilpotent cdga over $\G$, and a map $A\{n\} \rightarrow A$ of cdgas over $\G$ that is an isomorphism on $H^{i}(-)$ for $1 \leq i \leq n$ and an injection for $i = n + 1$.
\end{proof}
\begin{rmk} \label{min model preserve under qi}
If $f: A \to B$ is a quasi-isomorphism of cdgas over $\G$, and $s: A\{n\} \to A, \ t:B\{n\} \to B$ are n-minimal models, then there is an isomorphism of cdgas over $\G$: $g: A\{n\} \to B\{n\}$ such that $g \circ s$ is homotopic to $t\circ f$.  The proof is the same as the usual case in Chapter 4 of \cite{BG}. 
\end{rmk}

\section{The t-structure of \texorpdfstring{$\mathcal{D}^{G, f}_A$}{mathcal{D}}} \label{t}
The aim of this section is to define a $t$-structure on $\mathcal{D}_{A}^{\G, f}$ for $A$ which is a cohomologically connected cdga over $\G$. We recall the definition of the $t$-structure.
\begin{df} \label{df of t-structure}
A $t$-structure on a triangulated category $\mathcal{D}$ consists of essentially full subcategories $(\mathcal{D}^{\leq 0}, \mathcal{D}^{\geq 0})$ of $\mathcal{D}$ such that:
\begin{itemize}
\item
$\mathcal{D}^{\leq 0}[1] \subset \mathcal{D}^{\leq 0}, \mathcal{D}^{\geq 0}[-1] \subset \mathcal{D}^{\geq 0}$;
\item
$Hom_{\mathcal{D}}(M, N[-1]) = 0$ for $M \in \mathcal{D}^{\leq 0}$ and $N \in \mathcal{D}^{\geq 0}$;
\item
For every $M \in \mathcal{D}$, there is a distinguished triangle
$$M^{\leq 0} \rightarrow M \rightarrow M^{> 0} \rightarrow M^{\leq 0}[1]$$
with $M^{\leq 0} \in \mathcal{D}^{\leq 0}$ and $M^{> 0} \in \mathcal{D}^{\geq 0}[-1]$.
\end{itemize}
Write $\mathcal{D}^{\leq n}$ for $\mathcal{D}^{\leq 0}[-n]$ and $\mathcal{D}^{\geq n}$ for $\mathcal{D}^{\geq 0}[-n]$.
\

A $t$-structure $(\mathcal{D}^{\leq 0}, \mathcal{D}^{\geq 0})$ is non-degenerate if 
$$A \in \bigcap_{n \leq 0}\mathcal{D}^{\leq n} \ \ \ \ \ \text{and} \ \ \ \ \  B \in \bigcap_{n \geq 0}\mathcal{D}^{\geq n} \ \ \ \ \ \text{imply} \ \ \ \ \  A \cong B \cong 0.$$
\end{df}

There is a canonical augmentation $\epsilon: A \rightarrow \mathbb{Q}$, given by the projection onto $A^{0}\law 0 \raw = \mathbb{Q}$. So we have a functor:
$$q = \epsilon_{*}: \mathcal{CM}_{A}^{\G} \rightarrow \mathcal{M}_{\mathbb{Q}}^{\G}, \ \ \ \ \ \ \ q(M) = M \otimes_{A} \mathbb{Q}$$
and an exact tensor functor:
$$q: \mathcal{D}_{A}^{\G} \rightarrow \mathcal{D}_{\mathbb{Q}}^{\G}.$$
\begin{rmk} \label{conservative}
We recall that $\mathcal{D}_{\mathbb{Q}}^{\G, f}$ is the derived category of finite dimensional $\G$-representations. There is a canonical $t$-structure for $\mathcal{D}^{\G, f}_{\mathbb{Q}}$. We want to use $q$ to get the induced $t$-structure for $\mathcal{D}_{A}^{\G, f}$ when $A$ is a cohomologically connected cdga over $\G$. This is reasonable because the following general fact:
\

Let $\phi: A \rightarrow B$ be a map of cohomologically connected cdgas over $\G$. Then $\phi_{*}: \mathcal{D}_{A}^{\G, f} \rightarrow \mathcal{D}_{B}^{\G, f}$ is conservative, i.e., $\phi_{*}(M) \cong 0$ implies $M \cong 0$.
\begin{proof}
Take a non-zero object $M \in \mathcal{D}_{A}^{\G, f}$. Then we can find a cell module $P$ and a quasi-isomorphism $P \rightarrow M$ such that $W_{n-1}P = 0$, but $W_{n}P$ is not acyclic. We choose generating data $\{V_{j}, \phi_{j}\}_{j \in J}$ for $P$, such that $|V_{j}| \geq n$ for $j \in J$. Because $n$ is the minimal integer of the possible Adams degree, (like the proof of Lemma \ref{closed under summand}), we can know that $W_{n}P \otimes_{A} \mathbb{Q}$ is not acyclic. Notice that $W_{n}(P \otimes_{A} B) = W_{n}P \otimes_{A} B$ and $W_{n}P \otimes_{A} \mathbb{Q} = (W_{n}P \otimes_{A} B) \otimes_{B} \mathbb{Q}$. Therefore $P \otimes_{A} B$ is not isomorphic to zero in $\mathcal{KCM}_{B}^{\G}$ and $\phi_{*}M$ is non-zero in $\mathcal{D}^{\G, f}_{B}$.
\end{proof} 
\end{rmk}
\begin{rmk}
The inclusion $\mathbb{Q} \rightarrow A$ splits $\epsilon$. Then the functor $q$ defined above can be identified with the functor $gr^{W}_{*} = \prod_{n \in \mathbb{Z}}gr^{W}_{n}$. One can prove it by using the decomposition of the differential $d = d^{0} + d^{+}$, which is described in Lemma \ref{app} below and comparing two functors directly. 
\end{rmk}

Define full subcategories $\mathcal{D}_{A}^{\G, f, \leq 0}, \mathcal{D}_{A}^{\G, f, \geq 0}$ and $\mathcal{H}_{A}^{\G, f}$ of $\mathcal{D}_{A}^{\G, f}$.
$$\mathcal{D}_{A}^{\G, f, \leq 0} = \{M \in \mathcal{D}_{A}^{\G, f} | H^{n}(qM) = 0 \ for \ n > 0\}$$
$$\mathcal{D}_{A}^{\G, f, \geq 0} = \{M \in \mathcal{D}_{A}^{\G, f} | H^{n}(qM) = 0 \ for \ n < 0\}$$
$$\mathcal{H}_{A}^{\G, f} \ \ \ = \{M \in \mathcal{D}_{A}^{\G, f} | H^{n}(qM) = 0 \ for \ n \neq 0\}.$$
Then we have the following theorem as in \cite{KM, L}.
\begin{thm} \label{existence of t-structure}
Suppose $A$ is cohomologically connected. Then $$(\mathcal{D}_{A}^{\G, f, \leq 0}, \mathcal{D}_{A}^{\G, f, \geq 0})$$ is a non-degenerate $t$-structure on $\mathcal{D}_{A}^{\G, f}$ with heart $\mathcal{H}_{A}^{\G, f}$.
\end{thm}
\noindent $Proof.$ We can assume that $A$ is connected after replacing by its minimal model. The proof will divide into the following lemmas.
\begin{lem} \label{app}
Suppose that $A$ is connected. Let $M \in \mathcal{D}_{A}^{\G, f, \leq 0}$ (resp. $M \in \mathcal{D}_{A}^{\G, f, \geq 0}$). Then there is a cell $A$-module $P \in \mathcal{CM}^{\G, f}_{A}$ with generating data $\{V_{j}, \phi_{j}\}_{j \in J}$ such that $deg(\phi_{j}) \leq 0$ for all $j \in J$(resp. $deg(\phi_{j}) \geq 0$ for all $j \in J$), and a quasi-isomorphism $P \rightarrow M$.
\end{lem}
\begin{proof} We prove the case $M \in \mathcal{D}_{A}^{\G, f, \leq 0}$ only.
\

We choose a quasi-isomorphism $Q \rightarrow M$ with $Q \in \mathcal{CM}^{\G, f}_{A}$. Let $\{V_{j}, \phi_{j}\}_{j \in J}$ be generating data for $Q$. We can decompose $d_{Q}$ with two parts $d_{Q}^{0}$ and $d_{Q}^{+}$, where $d_{Q}^{0}$ maps $\phi_{j}(V_{j})$ to the submodule whose generating data $(\phi_{i}, V_{i})$ have the Adams degree $|V_{j}|$  and $d_{Q}^{+}$ map to the complement part. After choosing suitable generating data, we may assume the collection $S_{0}$ of $(V_{j}, \phi_{j})$ with $deg(\phi_{j}) = 0$ and $d_{Q}^{0}(\phi_{j}(V_{j})) = 0$ forms a basis of
$$ker(d^{0}: \oplus_{deg(\phi_{j}) = 0}\phi_{j}(V_{j}) \rightarrow \oplus_{deg(\phi_{i}) = 1}\phi_{i}(V_{i})).$$
Let $\tau^{\leq 0}Q$ be the sub $A$-module of $Q$ with the generating data of $S = \{(V_{j}, \phi_{j}) | deg(\phi_j) < 0 \} \bigcup S_{0}$.
\

Claim: $\tau^{\leq 0}Q$ is a subcomplex of $Q$. Consider
$$d_{Q}(\phi_{\alpha}(V_{\alpha})) = d_{Q}^{0}(\phi_{\alpha}(V_{\alpha})) \oplus d_{Q}^{+}(\phi_{\alpha}(V_{\alpha})).$$
Using the connected condition of $A$, we can know that:
\

1. If $\phi_{\beta}(V_{\beta}) \subset d_{Q}^{+}(\phi_{\alpha}(V_{\alpha}))$, then $deg(\phi_{\beta}) \leq deg(\phi_{\alpha})$.
\

\noindent Or 2. If $\phi_{\beta}(V_{\beta}) \subset d_{Q}^{0}(\phi_{\alpha}(V_{\alpha}))$, then $deg(\phi_{\beta}) = deg(\phi_{\alpha}) + 1$.
\

Consider $(V_{\alpha}, \phi_{\alpha}) \in S$ with $deg(\phi_{\alpha}) \leq -1$. Because $(d_{Q}^{0})^{2} = 0$, every summand of $d_Q^{0}(\phi_{\alpha}(V_{\alpha}))$ lies in $S_0$.  We only need to consider elements in $S_{0}$. Let $(V_{\alpha}, \phi_{\alpha}) \in S_{0}$. Then we have:
$$d_{Q}(\phi_{\alpha}(V_{\alpha})) \subset \bigoplus_{deg(\phi_{\beta}) = 0}\phi_{\beta}(A \otimes V_{\beta}) \oplus \bigoplus_{deg(\phi_{\gamma}) \leq -1}\phi_{\gamma}(A \otimes V_{\gamma}).$$
Using $d^{2}_{Q}(\phi_{\alpha}(V_{\alpha})) = 0$, we can know that $d^{0}(\phi_{\beta}(V_{\beta})) = 0$ for $deg(\phi_{\beta}) = 0$. This means that $d_{Q}(\phi_{\alpha}(V_{\alpha})) \subset \tau^{\leq 0}Q$. 
\

Next we can see that $\tau^{\leq 0}Q \rightarrow Q$ is a quasi-isomorphism. Using Remark \ref{conservative} , we need only to check that $q\tau^{\leq 0} Q \rightarrow qQ$ is a quasi-isomorphism. This is clear because of $qQ \cong qM$ and $M \in \mathcal{D}_{A}^{f, \leq 0}$.
\

For the case $M \in \mathcal{D}_{A}^{\G, f, \geq 0}$, we need to check the proof of Theorem \ref{App cell} carefully, where we can add the extra conditions on the degrees of the generating datum. See Lemma 1.6.2 in \cite{L}.
\end{proof}
\begin{lem}
Suppose that $A$ is connected. Then $Hom_{\mathcal{D}_{A}^{\G, f}}(M, N[-1]) = 0$ for $M \in \mathcal{D}_{A}^{\G, f, \leq 0}$ and $N \in \mathcal{D}_{A}^{\G, f, \geq 0}$.
\end{lem}
\begin{proof} By Lemma \ref{app}, we may assume that $M$ and $N[-1]$ are cell $A$-modules with the generating datum $\{(V_{\alpha}, \phi_{\alpha})\}_{deg(\phi_{\alpha}) \leq 0}$ and $\{(V_{\beta}, \phi_{\beta})\}_{deg(f_{\beta}) \geq 1}$. Recall we have:
$$Hom_{\mathcal{D}^{\G, f}_{A}}(M, N[-1]) = Hom_{\mathcal{KCM}^{\G, f}_{A}}(M, N[-1]).$$
If $\phi: M \rightarrow N[-1]$, then $deg(\phi) = 0$, which is impossible when we compute the degrees of both sides.
\end{proof}

\begin{lem}
Suppose that $A$ is connected. For $M \in \mathcal{D}_{A}^{\G, f}$, there is a distinguished triangle
$$M^{\leq 0} \rightarrow M \rightarrow M^{> 0} \rightarrow M^{\leq 0}[1]$$
with $M^{\leq 0} \in \mathcal{D}^{\G, f, \leq 0}$ and $M^{> 0} \in \mathcal{D}^{\G, f, \geq 0}[-1]$.
\end{lem}
\begin{proof} Following the same proof of Lemma \ref{app}, we can get a sub cell $A$-module $\tau^{\leq 0}M$ of $M$ such that:
\begin{itemize}
\item
 $\tau^{\leq 0}M$ have generating data $\{(V_{\alpha}, \phi_{\alpha})\}_{deg(\phi_{\alpha}) \leq 0}$.
\item
The map $q \tau^{\leq 0}M \rightarrow q M$ gives an isomorphism on $H^{n}$ for $n \leq 0$.
\end{itemize}

Let $M^{\leq 0} = \tau^{\leq 0}M$ and let $M^{> 0}$ be the cone of $\tau^{\leq 0}M \rightarrow M$. This gives us the distinguished triangle in $\mathcal{D}_{A}^{\G}$:
$$M^{\leq 0} \rightarrow M \rightarrow M^{> 0} \rightarrow M^{\leq 0}[1].$$
Because $M \in \mathcal{D}_{A}^{\G, f}$, then $gr^{W}_{n}M \in D^{\G, f}_{\mathbb{Q}}$ for all $n$ and is isomorphic to zero for all but finitely many $n$. This means $M^{\leq 0}$ and $M^{> 0}$ all satisfy these two conditions by using the long exact sequence. Recall the distinguished triangle given by the weight filtration $gr_{n}^{W}M \rightarrow M \rightarrow M^{> n} \rightarrow gr_{n}^{W}M[1]$. By induction, we can show that $M^{\leq 0}$ and $M^{> 0}$ are all in $\mathcal{D}^{\G, f}_{A}$. After applying the functor $q$, we can know that $M^{\leq 0} \in \mathcal{D}^{\G, f, \leq 0}$ and $M^{> 0} \in \mathcal{D}^{\G, f, \geq 0}[-1]$.
\

The only thing need to check is non-degenerate for the $t$-structure. If we take $M \in \bigcap_{n \leq 0}\mathcal{D}^{\leq n}$, then $H^{n}(qM) = 0$ for all $n$, i.e., $qM \cong 0$ in $\mathcal{D}^{\G, f}_{\mathbb{Q}}$. By the conservative property of the functor $q$, we know that $A \cong 0$ in $\mathcal{D}^{\G, f}_{A}$. Another case is similar.
\end{proof}

\begin{rmk}
The key point for the proof is that, for $M \in \mathcal{D}^{\G, f}_{A}$, we have a lowest bound for the Adams degree.
\end{rmk}
We recall the definition of the neutral Tannakian category.
\begin{df} \label{def of Tan}
A neutral Tannkian category over $\mb{Q}$ is a rigid abelian tensor category $(C, \otimes)$ such that $\mb{Q} = End(1)$ for which there exists an exact faithful $k$-linear tensor functor $\omega: C \rightarrow Vec_{\mb{Q}}$. Here $1$ means the identity object under tensor product. The functor $\omega$ is called the fiber functor.
\end{df}
\begin{rmk}
As shown in Theorem 2.11 of \cite{DM}, every neutral Tannakian category is equivalent to the category of finite-dimensional representations of an affine group scheme. This affine group scheme is called the Tannakian fundamental group of $(C, \omega)$.
\end{rmk}
\begin{prop}
$\mathcal{H}_{A}^{\G, f}$ is a neutral Tannakian category over $\mathbb{Q}$.
\end{prop}
\begin{proof} The derived tensor product makes $\mathcal{H}^{\G, f}_{A}$ into an abelian tensor category. First we give a description about $\mathcal{H}^{\G, f}_{A}$.
\begin{lem}
$\mathcal{H}^{\G, f}_{A}$ is the smallest abelian subcategory of $\mathcal{H}^{\G, f}_{A}$ containing the objects $A \otimes V$, where $V$ is any rational $\G$-representation, and closed under extensions in $\mathcal{H}^{\G, f}_{A}$.
\end{lem}
\begin{proof} (Induction on the weight filtration.) Let $\mathcal{H}^{\G, T}_{A}$ be the full abelian subcategory containing all the objects $A \otimes V$, where $V$ is any rational $\G$-representation, and closed under extensions in $\mathcal{H}^{\G, f}_{A}$. Let $M \in \mathcal{H}^{\G, f}_{A}$ and $N = min \{n | W_{n}M \neq 0 \}$. Then we have an exact sequence:
$$0 \rightarrow gr^{W}_{N}M \rightarrow M \rightarrow W^{> N}M \rightarrow 0.$$
By Lemma \ref{closed under summand}, we have $gr^{W}_{N}M \cong A \otimes C$, where $C$ is in $D^{b}(\G)$. Because the category of representations of $\G$ is semisimple, we can think $C$ as a direct sum of rational $\G$-representations with some shifts. Assume there exists a summand $W[i]$ of $C$ with shift $i \neq 0$. Then applying $q$, we get that $0 \neq H^{i}(q(gr^{W}_N M)) \subset H^{i}(qM)$, which is a contradiction of our choice of $M \in\mathcal{H}^{\G, f}_{A} $. This implies that $gr^{W}_{N}M \in \mathcal{H}^{\G, T}_{A}$. By induction on the length of the weight filtration, $W^{> N}M$ is in $\mathcal{H}^{\G, T}_{A}$. So $M \in \mathcal{H}^{\G, T}_{A}$ and $\mathcal{H}^{\G, T}_{A} = \mathcal{H}^{\G, f}_{A}$.
\end{proof}

Since $(A \otimes V)^{\vee} = A \otimes V^{\vee}$, where $V^{\vee}$ is the dual representation of $V$, it follows from the above description that $M \rightarrow M^{\vee}$ restricts from $\mathcal{D}^{\G, f}_{A}$ to an exact involution on $\mathcal{H}^{\G, f}_{A}$. $\mathcal{H}^{\G, f}_{A}$ is rigid because $\mathcal{D}^{\G, f}_{A}$ is rigid.
\

The identity for the tensor product is $A$ and $\mathcal{H}^{\G, f}_{A}$ is $\mathbb{Q}$-linear.
\

We have a rigid tensor functor $q: \mathcal{H}^{\G, f}_{A} \rightarrow \mathcal{H}^{\G, f}_{\mathbb{Q}}$. Notice that $\mathcal{H}^{\G, f}_{\mathbb{Q}}$ is equivalent to the category of rational representations of $\G$. So there is a faithful forgetful functor $w: \mathcal{H}^{\G, f}_{\mathbb{Q}} \rightarrow Vec_{\mathbb{Q}}$. We need only to see that $q$ is faithful.
\

Recall we can identify $q$ with $gr^{W}_{*} = \oplus gr^{W}_{n}$. Let $f: M \rightarrow N$ be a map in $\mathcal{H}^{f}_{A}$ such that $gr^{W}_{n}(f) = 0$ for all $n$. We need to show that $f = 0$. Again do the induction on the length of the weight filtration. We may assume that $W^{n}f = 0$, where $n$ is the minimal integer such that $W_{n}M \oplus W_{n}N \neq 0$. Thus $f$ is given by a map
$$\tilde{f}: W^{> n}M \rightarrow gr^{W}_{n}N.$$
Claim: $\tilde{f} = 0$. 
\

Using the induction on the weight filtration, we need to show the following statement:
\

Given $V$ and $W$ pure weight rational $\G$-representations such that $|V| > |W|$, then we have:
$$Hom_{\mathcal{H}^{\G, f}_{A}}(A \otimes V, A \otimes W) \cong 0.$$
Firstly we have:
\begin{equation} \nonumber
\begin{split}
& Hom_{\mathcal{H}^{\G, f}_{A}}(A \otimes V, A \otimes W) \cong Hom_{\ml{D}^{\G, f}_{A}}(A \otimes V, A \otimes W) \\
\cong & Hom_{D(\G)}(V, A \otimes W) \cong Hom_{\G}(\mb{Q}, H^0(A \otimes W \otimes V^{\vee})).
\end{split}
\end{equation}
Because $A$ is connected, $H^0(A \otimes W \otimes V^{\vee}) \cong W \otimes V^{\vee}$, which is a rational representation of $\G$ with Adams degree strictly smaller than zero. This implies that:
$$Hom_{\G}(\mb{Q}, H^0(A \otimes W \otimes V^{\vee})) \cong 0.$$
Therefore we get that $q$ is faithful.
\end{proof}

\section{The bar construction} \label{bar con}
Let $A$ be a cdga over $\G$ and let $M, N$ be two dg $A$- modules. Then we define:
$$T^{\G}(N, A, M) = N \otimes T(A) \otimes M,$$
where $T(A) = \mathbb{Q} \oplus A \oplus (A \otimes A) \oplus \cdots = \oplus_{r \geq 0}T^{r}(A)$ is the tensor algebra. It is spanned by the elements of the form $n[a_{1}| \cdots |a_{r}]m$. Notice that $T^{\G}(N, A, M)$ is a simplicial graded abelian group with $N \otimes T^{r}(A) \otimes M$ in degree $r$, whose face maps are:
$$\delta_{0}(n[a_{1}| \cdots |a_{r}]m) = n a_{1}[a_{2}| \cdots |a_{r}]m,$$
$$\delta_{i}(n[a_{1}| \cdots  |a_{r}]m) = n a_{1}[a_{2}| \cdots |a_i a_{i+1}| \cdots | a_{r}]m, \ 1 \leq i \leq r-1$$
$$\delta_{r}(n[a_{1}| \cdots |a_{r-1}] a_{r} m) = n [a_{1}|\cdots |a_{r-1}]a_rm,$$
and degeneracies are:
$$s_{i}(n[a_{1}| \cdots |a_{r}]m) = n[a_{1}| \cdots |a_{i-1}| 1 |a_{i}| \cdots |a_{r}]m.$$
Define:
$$\delta = \sum_{0 \leq i \leq r}(-1)^{i}\delta_{i}: N \otimes T^{r}(A) \otimes M \rightarrow N \otimes T^{r - 1}(A) \otimes M$$
Let $D^{\G}(N, A, M)$ be the degenerate elements, those elements are spanned by the images of the $s_{i}$ for every $i$. 
\begin{df}
Define the bar complex of $M$ and $N$ to be:
$$B^{\G}(N, A, M) = T^{\G}(N, A, M)/D^{\G}(N, A, M).$$
\end{df}
Note that $B^{\G}(N, A, M)$ is a bicomplex. The total differential is defined by
\begin{equation} \nonumber
\begin{split}
& d(n[a_{1}| \cdots |a_{r}]m) \\
=  & \partial(n[a_{1}| \cdots |a_{r}]m)) + (-1)^{deg(n) + deg(m) + \sum deg(a_{i})}\delta(n[a_{1}| \cdots |a_{r}]m),
\end{split}
\end{equation}
where $\partial$ denotes the usual differential of $A$.
We will consider the following special case that $M = N = \mathbb{Q}$, which is denoted by $\bar{B}^{\G}(A)$, called the reduced bar construction. We collect formal properties of $\bar{B}^{\G}(A)$.
\begin{itemize}
\item
(Shuffle product) $\cup: \bar{B}^{\G}(A) \otimes \bar{B}^{\G}(A) \rightarrow \bar{B}^{\G}(A)$
$$[a_{1}|\cdots|a_{p}] \cup [a_{p+1}|\cdots|a_{p+q}] = \sum sgn(\sigma)[x_{\sigma(1)}|\cdots|x_{\sigma(p+q)}]$$
where the sum is over all $(p,q)$ shuffles $\sigma \in \Sigma_{p+q}$ (recall $\Sigma_{p+q}$ is the symmetric group on $p+q$ letters) and the sign of $\sigma$ is taking into account the degree of $a_{i}$.
\item
(Coproduct) $\Delta: \bar{B}^{\G}(A) \rightarrow \bar{B}^{\G}(A) \otimes \bar{B}^{\G}(A)$
$$\Delta([a_{1}|\cdots|a_{n}]) = \sum_{i = 0}^{n}(-1)^{i (deg(a_{i+1}) + \cdots + deg(a_{n}))}[a_{1}|\cdots|a_{i}] \otimes [a_{i+1}|\cdots|a_{n}].$$
\item
(involution) $\iota: \bar{B}^{\G}(A) \rightarrow \bar{B}^{\G}(A)$
$$\iota([a_{1}|\cdots|a_{n}]) = (-1)^{m} [a_{n}|a_{n-1}| \cdots |a_{1}], m = \sum_{1 \leq i < j \leq n} deg(a_{i})deg(a_{j}).$$
\end{itemize}
These properties make $\bar{B}^{\G}(A)$ a graded-commutative differential graded Hopf algebra in the category of $\G$-representations. So $H^{0}(\bar{B}^{\G}(A))$ is a commutative Hopf algebra over $\G$. If consider the Adams grading structure, we can see that $\bar{B}^{\G}(A)$ also has the Adams grading structure, and $\chi_{A} = H^{0}(\bar{B}^{\G}(A))$ is an Adams graded Hopf algebra over $\G$ (or a graded Hopf algebra object in $\mathbf{Rep}_{\G}$).
\begin{df} \label{gamma}
Define $\gamma_{A} = I_{\chi_{A}}/(I_{\chi_{A}})^{2}$, where $I_{\chi_{A}}$ is the augmentation ideal of $\chi_{A}$.
\end{df}
\begin{lem}
$\gamma_{A}$ determines a structure of a cdga over $\G$.
\end{lem}
\begin{rmk}
Recall the definition of co-Lie algebras firstly. A co-Lie algebra is a $k$-module $\gamma$ with a cobracket map $\gamma \rightarrow \gamma \otimes \gamma$ such that the dual $\gamma^{\vee}$ is a Lie algebra via the dual homomorphism. Sullivan showed the following statement (Lemma 2.7 in \cite{KM} or p.279 in \cite{Sul}):
\

A co-Lie algebra $\gamma$ determines and is determined by a structure of DGA on $\wedge(\gamma[-1])$.
\end{rmk}
\begin{proof} The coproduct $\Delta$ of $\bar{B}^{\G}(A)$ induce a coproduct on $\chi_{A}$, denoted also by $\Delta$, satisfying:
$$\Delta(x) = x \otimes 1 + 1 \otimes x$$
for $x \in I_{\chi_{A}}$. (And also the involution.) Then $\Delta - \iota \Delta$ gives the cobracket on $\chi_{A}$. So is $\gamma_{A}$ and the quotient map $I_{\chi_{A}} \rightarrow \gamma_{A}$ is a map of co-Lie algebras. By the remark above, we know $\gamma_{A}$ determines a cdga structure $\wedge(\gamma_{A}[-1])$. This structure is compatible with the $\G$-representation structure. Notice that $I_{\chi_{A}}$ does not have the Adams degree $0$ part. Then $\wedge(\gamma_{A}[-1]) \law 0 \raw = \mathbb{Q}$. So $\wedge(\gamma_{A}[-1])$ is a cdga over $\G$.
\end{proof}
\begin{lem} \label{Bar preserve under quasi-iso}
Let $A$ be a cdga over $\G$. Then $H^{*}(\bar{B}^{\G}(A))$ and $\chi_{A}$ is functorial in $A$ and is a quasi-isomorphism invariant in $A$.
\end{lem}
\begin{proof} Use the Eilenberg-Moore spectral sequence. See Lemma 2.21 in \cite{BK}.
\end{proof}
\begin{thm} \label{1-minimal model}
Let $A$ be a cohomologically connected cdga over $\G$. Then the $1$-minimal model $A\{1\}$ of $A$ is isomorphic to $\wedge(\gamma_{A}[-1])$.
\end{thm}
\begin{proof} Follow the proof of \cite{BK}. From Lemma \ref{Bar preserve under quasi-iso}, we can assume that $A$ is a generalized nilpotent cdga over $\G$. A generalize nilpotent cdga $A$ over $\G$ is a direct limit $(A_{\alpha})$ of nilpotent cdga's.
So we can assume that $A$ is a nilpotent cdga over $\G$ with a free generator $E$ which is a complex of $\G$-representations. We need to use the following lemma, whose proof is totally the same as Lemma 2.32 in \cite{BK}.
\begin{lem} \label{ss}
Let $A$ be as above with free generator $\G$-representation $V$. Fix an integer $s > 0$. Consider the decreasing filtration on $\bar{B}^{\G}(A)$,
$$F^{k}\bar{B}^{\G}(A) = \langle x_{11} \cdots x_{1n_{1}} \otimes \cdots \otimes x_{m1}\cdots x_{mn_{m}} | $$
$$s\sum deg x_{ij} + s\sum n_{j} - (2s-1)m \geq k \rangle.$$
Then, for a sufficiently large $s$ depending on $A$, the resulting spectral sequence satisfies
$$\wedge(V[1]) \cong E_{2s} \cong E_{\infty} \cong Gr_{F}H^{*}(\bar{B}^{\G}(A)).$$
\end{lem}
\noindent Assuming this lemma, the projection map $\bar{B}^{\G}(A) \rightarrow A$ induces a map
$$\phi: QH^{*}(\bar{B}^{\G}(A)) \rightarrow (QA)[1].$$
This is because that the boundaries and decomposable elements map to decomposable elements. In fact, boundaries are the sum of the form $\partial(a)$ and $a_1 \cdot a_2$. Using that the generalized nilpotence implies the minimal property (Proposition \ref{generalized nil equi minimal}), both of these elements are decomposable. By the above lemma, this is an isomorphism. Furthermore, $\phi$ is an isomorphism between co-Lie algebras.
\
 
Restricting $\phi$ to the degree $0$ part, we know that $$\wedge(\gamma_{A}[-1]) = \wedge (QH^0(\bar{B}^{\G}(A))[-1])$$ is isomorphic to $1$-minimal model of $A$.
\end{proof}

\section{Alternative identifications of the category \texorpdfstring{$\mathcal{H}_{A}^{f}$}{mathcal{H}}} \label{heart}
In Proposition \ref{generalized nil equi minimal}, we show that the connected generalized nilpotent cdga over $\G$ can be recognized as a connected minimal cdga over $\G$. Similarly we can also define the minimal cell $A$-module.
\begin{df}
An Adams degree bounded below cell $A$-module is minimal if it is almost free and $d(M) \subset (IA)M$.
\end{df}
\begin{df}
Let $M$ be an Adams degree bounded below $A$-module. We define the nilpotent filtration $\{F_{t}M\}$ by letting $F_{0}M = 0$ and inductively letting $F_{t}M$ be the sub $A$-module generated by $F_{t-1}M \cup \{m | dm \in F_{t-1}M\}$.
\end{df}
\begin{rmk}
The minimal cell modules have the similar properties as the connected minimal cdgas. We can also define the generalized nilpotent $A$-modules. Because the proof of the following properties is the same as Part IV, section 3 in \cite{KM}, we only list the main properties.
\begin{itemize}
\item
A bounded below $A$-module $M$ is generalized nilpotent if and only if it is a minimal cell $A$ module.
\item
Let $N$ be a dg $A$-module. Then there is a quasi-isomorphism $e: M \rightarrow N$, where $M$ is a minimal $A$-module. This is unique up to the homotopy.
\end{itemize}
\end{rmk}
Next, we want to use another way to describe cell $A$-modules, which is called the connection matrix. See \cite{KM}(namely the twisting matrix) or \cite{L}.
\begin{df}
Let $(M, d_{M})$ be a complex of $\G$-representations. An $A$-connection for $M$ is a map
$$\Gamma: M \rightarrow IA \otimes M$$
of $\G$ representations and cohomological degree $1$. We say $\Gamma$ is flat if
$$d\Gamma + \Gamma^{2} = 0.$$
Here $d\Gamma = d_{IA \otimes M} \circ \Gamma + \Gamma \circ d_{M}$ and we extend $\Gamma$ to
$$\Gamma: IA \otimes M \rightarrow IA \otimes M$$
by the Leibniz rule.
\end{df}
\begin{rmk}
Given a connection $\Gamma: M \rightarrow IA \otimes M$, we define
$$d_{0}: M \rightarrow A \otimes M = M \oplus IA \otimes M, m \rightarrow d_{M}m \oplus \Gamma m$$
and extend $d_{0}$ to $d_{\Gamma}: A \otimes M \rightarrow A \otimes M$ by the Leibniz rule. The above equation is equivalent to saying that $d_{\Gamma}^{2} = 0$.
\end{rmk}
\begin{df}
We call an $A$-connection $\Gamma$ for $M$ nilpotent if $M$ admits a filtration by complexes of $\G$-representations:
$$0 = M_{-1} \subset M_{0} \subset \cdots \subset M_{n} \subset \cdots \subset M$$
such that $M = \cup_{n} M_{n}$ and such that
$d_{M}(M_{n}) \subset M_{n-1}$ and $\Gamma(M_{n}) \subset IA \otimes M_{n-1}$
for every $n \geq 0$.
\end{df}
\begin{rmk}
Let $\Gamma: M \rightarrow IA \otimes M$ be a flat nilpotent connection. Then the dg $A$-module $(A \otimes M, d_{\Gamma})$ is a cell module.
\end{rmk}
\begin{lem}
Let $\Gamma: M \rightarrow IA \otimes M$ be a flat connection. Suppose there is an integer $r_{0}$ such that $|m| \geq r_{0}$ for all $m \in M$. Then $\Gamma$ is nilpotent.
\end{lem}
\begin{proof}
The proof is the same as Lemma 1.13.3 in \cite{L}.
\end{proof}
\begin{df}
A morphism $f: (M, d_{M}, \Gamma_{M}) \rightarrow (N, d_{N}, \Gamma_{N})$ is a map of complexes of $\G$-representations:
$$f = f_{0} + f^{+}: M \rightarrow A \otimes N = N \oplus IA \otimes N$$
such that $d_{\Gamma_{N}}f = fd_{\Gamma_{M}}$.
\end{df}
\begin{df}
We denote the category of flat nilpotent connections over $A$ by $Conn^{\G}_{A}$ and denote the full subcategory of flat nilpotent connections on $M$ with $M$ a bounded complex of rational $\G$-representations by $Conn_{A}^{\G, f}$.
\end{df}
We can define a tensor operation on $Conn_{A}$ by
$$(M, \Gamma) \otimes (M^{'}, \Gamma^{'}) = (M \otimes M^{'}, \Gamma \otimes id + id \otimes \Gamma^{'}).$$
Complexes of $\mathbb{Q}$-vector spaces act on $Conn_{A}$ by:
$$(M ,\Gamma) \otimes K = (M , \Gamma) \otimes (K, 0).$$
We recall that $I$ is the complex
$$\mathbb{Q} \xrightarrow{\delta} \mathbb{Q} \oplus \mathbb{Q}$$
with $\mathbb{Q}$ in degree $-1$ and with connection $0$. We have the two inclusions $i_{0}, i_{1}: \mathbb{Q} \rightarrow I$.
\begin{df}
Two maps $f, g: (M, \Gamma) \rightarrow (M^{'}, \Gamma^{'})$ are homotopic if there is a map $h: (M, \Gamma) \otimes I \rightarrow (M^{'}, \Gamma^{'})$ satisfying $f = h \circ (id \otimes i_{0}), g = h \circ (id \otimes i_{1})$.
\end{df}
\begin{df}
Denote the homotopy category of $Conn^{\G}_{A}$ by $\mathcal{H}Conn^{\G}_{A}$, which has the same objects as $Conn^{\G}_{A}$ and morphisms are homotopy classes of maps in $Conn^{\G}_{A}$.
\end{df}
\begin{rmk}
When we pass to homotopy classes and given a cell $A$-module $M$, it is totally determined by the underlying $\G$-representation $M_{0}$, i.e., $M_{0} = M \otimes_{A} \mathbb{Q}$.
\end{rmk}
We list the main properties of flat connections, and the proof is the same as in Section 1.14 in \cite{L}.
\begin{itemize}
\item
The category of $A$-cell modules is equivalent to the category of flat nilpotent $A$-connections.
\item
The above equivalence passes to an equivalence of $\mathcal{H}Conn^{\G}_{A}$ with the homotopy category $\mathcal{KCM}^{\G}_{A}$ as triangulated tensor categories.
\item
Suppose that $A$ is connected. The second equivalence defines an equivalence of Tannakian categories $\mathcal{H}_{A}^{\G, f}$ and the category of flat connections on $\G$-representations $Conn^{\G, f}_{A}$.
\end{itemize}
Given $A$ a cohomologically connected cdga over $\G$, by Theorem \ref{1-minimal model}, we know that: 
$$\gamma_A = A\{1\}^1.$$
Assume that $A$ is a generalized nilpotent dga over $\G$, which implies that $A\{1\} \cong A$. Then the co-Lie algebra structure of $\gamma_A$ is given by the restriction of $d$ to $A^1$. Notice that, by the minimal property (Proposition \ref{generalized nil equi minimal}), $d$ factors through:
$$d: A^1 \to \wedge^2 A^1 \subset A^2.$$
Let $M$ be a complex of rational $\G$-representations and $\Gamma: M \rightarrow IA \otimes M$ is a flat connection. In fact, $\Gamma$ is a map $$\Gamma; M \rightarrow A^{1} \otimes M$$ and the flatness is just saying that $\Gamma$ makes $M$ into an Adams graded co-module for the co-Lie algebra $\gamma_{A}$ over $\G$. In fact, we have  equivalences between categories:
$$\mathcal{H}_{A}^{\G, f} \cong Conn^{\G, f}_{A} \cong co-rep^{\G, f}(\gamma_{A}).$$
\section{The main theorem} \label{main}
\begin{lem}  \label{cri for equi between tri}
Let $\mathcal{D}$ be a triangulated category with t-structure $(\mathcal{D}^{\leq 0}, \mathcal{D}^{\geq 0})$. We denote its heart by $\mathcal{H}$. Assume that there is a triangulated functor $\rho: D^{b}(\mathcal{H}) \rightarrow \mathcal{D}$ such that:
\begin{itemize}
\item
$\rho |_{\mathcal{H}[i]}$ is an inclusion for any $i \in \mathbb{Z}$;
\item
$\mathcal{D}$ is bounded, i.e., for any $M \in \mathcal{D}$, there exist $a \leq b \in \mathbb{Z}$ satisfying $M \in \mathcal{D}^{[a, b]} = \mathcal{D}_{\geq a} \cap \mathcal{D}_{\leq b}$.
\item
For any $M, N \in \mathcal{H}$ and $n \in \mathbb{Z}$, $\rho$ induces an isomorphism
$$Hom_{D^{b}(\mathcal{H})}(M, N[n]) \xrightarrow{\sim} Hom_{\mathcal{D}}(\rho(M), \rho(N)[n]).$$
Then $\rho$ is an equivalence between triangulated categories.
\end{itemize}
\end{lem}
\begin{proof} We do the induction on the length of the object. Given an object $A$ in $\mathcal{D}$, there exist the minimal $a$ and maximal $b$ such that $A \in \mathcal{D}_{\geq a} \cap \mathcal{D}_{\leq b}$. Then we define the length of $A$ to be $b-a$. Firstly, we prove the following:
\

For any $A, B \in D^{b}(\mathcal{H})$ and $n \in \mathbb{Z}$, we have:
\begin{equation} \label{star}
Hom_{D^{b}(\mathcal{H})}(M, N[n]) \xrightarrow{\sim} Hom_{\mathcal{D}}(\rho(M), \rho(N)[n]). 
\end{equation}

By induction, we assume that, for any $A \in D^{b}(\mathcal{H})^{a,b}$, $B \in D^{b}(\mathcal{H})^{c,d}$ and $max\{b-a, d-c\} \leq m-1$, the above is true. Take any $A$ with length smaller than $m$, and $B$ with length $m = b - a$. There is a distinguished triangle:
$$\tau_{\geq a}\tau_{\leq b-1}B \rightarrow B \rightarrow \tau_{\geq b}\tau_{\leq b}B \rightarrow \tau_{\geq a}\tau_{\leq b-1}B[1] \rightarrow .$$
Then we have a long exact sequence:
$$Hom(A, \tau_{\geq a}\tau_{\leq b-1}B) \rightarrow Hom(A, B) \rightarrow Hom(A, \tau_{\geq b}\tau_{\leq b}B) $$
$$\rightarrow Hom(A, \tau_{\geq a}\tau_{\leq b-1}B[1]) \rightarrow \cdots.$$
Using functoriality to compare the above sequence with:
$$Hom(\rho(A), \rho(\tau_{\geq a}\tau_{\leq b-1}B)) \rightarrow Hom(\rho(A), \rho(B)) \rightarrow Hom(\rho(A), \rho(\tau_{\geq b}\tau_{\leq b}B)) $$
$$\rightarrow Hom(\rho(A), \rho(\tau_{\geq a}\tau_{\leq b-1}B[1])) \rightarrow \cdots.$$
We know (\ref{star}) is holding for $A, B$ by the five lemma and induction. Then we assume both $A$ and $B$ have length $m$. Using the similar method and induction again, we can know that (\ref{star}) is holding, i.e., $\rho$ is fully faithful.
\

Next we want to use induction to show that $\rho$ is essentially surjective. It is enough to show that, for any object $B \in \mathcal{D}$, there exists $A \in D^{b}(\mathcal{H})$ such that $\rho(A) \cong B$. Take any element $B \in \mathcal{D}$ with length $m$. Then we have:
$$\tau_{\geq a}\tau_{\leq b-1}B \rightarrow B \rightarrow \tau_{\geq b}\tau_{\leq b}B \rightarrow \tau_{\geq a}\tau_{\leq b-1}B[1] \rightarrow .$$
In other words, we have:
$$\tau_{\geq b}\tau_{\leq b}B[-1] \xrightarrow{f} \tau_{\geq a}\tau_{\leq b-1}B \rightarrow B \rightarrow \tau_{\geq b}\tau_{\leq b}B \rightarrow .$$
By assumption, we have $A_{1}$ and $A_{2} \in D^{b}(\mathcal{H})$ map to $\tau_{\geq b}\tau_{\leq b}B[-1]$ and $\tau_{\geq a}\tau_{\leq b-1}B$ respectively. By (\ref{star}), we know that there exists a map $g$ from $A_{1}$ to $A_{2}$, whose image under $\rho$ is just $f$. We take $A = cone(g)$. Then by the axiom of triangulated categories, there exists a map $\rho(A) \rightarrow B$. By the five lemma and Yoneda lemma, applying the functor of type $Hom(\tilde{B}, )$, where $\tilde{B} \in \mathcal{D}$, we know that $\rho(A) \cong B$. 
\end{proof}
\begin{thm}  \label{main thm for cdga over $\G$}
Let $A$ be a cohomologically connected cdga over $\G$. Then
\begin{itemize}
\item
There is a functor:
$$\rho: D^{b}(\mathcal{H}_{A}^{\G, f}) \longrightarrow \mathcal{D}^{\G, f}_{A}.$$
\item
The functor $\rho$ constructed above is an equivalence of triangulated categories if and only if $A$ is $1$-minimal.
\end{itemize}
\end{thm}
\begin{proof} We first need to construct a functor $\rho$. Consider a bounded complex
$$M^{*} = \{ M^{n}, \delta^{n}: M^{n} \rightarrow M^{n+1}\}$$
in $\mathcal{H}_{A}^{\G, f}$. Assume that each $M^{n}$ is minimal. Furthermore, we assume that it is given by the generating data $\{ V_{j^{n}}, \phi_{j^{n}} \}_{j^{n} \in J^{n}}$ and the connection matrix $\Gamma^{n}$. Then we define $\rho M^{*}$ with its generating data $\{ V_{j^{n}}, \phi_{j^{n}}[n] \}_{j^{n} \in J^{n}}$ and its differential given by:
$$d |_{\phi_{j^{n}}[n](V_{j^{n}})}  = \Gamma^{n}[n] + \delta^{n}[n].$$
If $f^{*}: M^{*} \rightarrow N^{*}$ is a quasi-isomorphism of chain complexes, then $\rho(f^{*})$ is a quasi-isomorphism of $A$-modules.  
\

Let's prove the second statement now. We assume that $A$ is $1$-minimal, i.e., $A \cong \wedge^{*} (\gamma[-1])$, where $\gamma$ is the co-Lie algebra consisted by the indecomposable elements of $H^{0}(\bar{B}^{\G}(A))$.

In order to apply the above result to our case ($\mathcal{D} = \mathcal{D}^{\G, f}_{A}$ and $\mathcal{H} = \mathcal{H}_{A}^{\G, f}$), we need to check the conditions in Lemma  \ref{cri for equi between tri}. The first and second condition are automatic. Let's check the third condition.
\

Notice that $\mathcal{H}_{A}^{\G, f}$ can be identified with the category of co-representations of $\gamma$ in the category of $\G$-representations. In fact, given a finite dimensional co-representation $V$, we can associate it with a cell module $A \otimes V$.
\

We recall the following basic facts. (Lemma 23.1, Example 1(p.315) and p.319, 320 in\cite{FHT}.) 
Given a differential graded Lie algebra $L$, we have:
$$Ext^{n}_{L}(\mathbb{Q}, \mathbb{Q}) \cong Ext^{n}_{UL}(\mathbb{Q}, \mathbb{Q}) \cong H^{n}((\wedge^{*}(L[-1]))^{\vee}),$$
and
$$Ext^{n}_{L}(\mathbb{Q}, V) \cong Ext^{n}_{UL}(\mathbb{Q}, V) \cong H^{n}((\wedge^{*}(L[-1]))^{\vee} \otimes V),$$
where $\vee$ means taking the dual, $UL$ is the universal enveloping Lie algbera of $L$ and $V$ is any $L$-module.  $L[-1]_{k} = L_{k-1}$.
\

Applying to the co-Lie algbebra $\gamma$, we get:
$$Ext^{n}_{\gamma}(\mathbb{Q}, \mathbb{Q}) \cong H^{n}(\wedge^{*}(\gamma[-1])).$$
In fact, the proof of this isomorphism can be extended to the following case.
\

Given a co-Lie algebra $\gamma$ over $\G$ and a $\gamma$ co-representation $V$, we have:
$$Ext^{n}_{\gamma}(\mathbb{Q}, V) \cong H^{n}(\wedge^{*}(\gamma[-1]) \otimes V)\law 0 \raw.$$  
Notice that the left hand side computes the extension groups in the category of $\gamma$-representations. We have:
 $$H^{n}(\wedge^{*}(\gamma[-1]) \otimes V)\law 0 \raw = Hom_{\G}(\mb{Q}, H^{n}(\wedge^{*}(\gamma[-1]) \otimes V)).$$ 
Therefore we get:
\begin{equation} \nonumber
\begin{split}
& Hom_{D^{b}(\mathcal{H}_{A}^{\G, f})}(\wedge^{*}(\gamma[-1]) \otimes V, \wedge^{*}\gamma[-1] \otimes W[n])  \\
\cong & Ext^{n}_{\mathcal{H}_{A}^{\G, f}}(\wedge^{*}(\gamma[-1]) \otimes V, \wedge^{*}(\gamma[-1]) \otimes W) \\
\cong &  Ext^{n}_{\gamma}(\mathbb{Q}, V^{\vee} \otimes W) \\
\cong & H^{n}(\wedge^{*}(\gamma[-1]) \otimes V^{\vee} \otimes W)\law 0 \raw \\
\cong & H^{0}(\wedge^{*}(\gamma[-1]) \otimes V^{\vee} \otimes W[n])\law 0 \raw \\
\cong & Hom_{\mathcal{D}^{\G, f}_{A}}(\mathbb{Q}, \wedge^{*}(\gamma[-1]) \otimes V^{\vee} \otimes W[n]) \\
\cong & Hom_{\mathcal{D}^{\G, f}_{A}}(\wedge^{*}(\gamma[-1]) \otimes V, \wedge^{*}(\gamma[-1]) \otimes W[n]).
\end{split}
\end{equation}
It's easy to check that the composition of these isomorphisms is given by $\rho: D^{b}(\mathcal{H}_{A}^{\G, f}) \longrightarrow \mathcal{D}^{\G, f}_{A}.$
\

Conversely, we assume that $\rho$ is an equivalence. Without loss of generality, we assume that $A$ is generalized nilpotent. The above computation tell us that $H^{n}(A \otimes V)\law 0 \raw \cong Ext^{n}_{\gamma}(\mathbb{Q}, V)$, where $\gamma$ is the co-Lie algebra consisting of the indecomposable elements in $H^{0}(\bar{B}(A))$. Let us consider the map $A \rightarrow \wedge^{*}(\gamma[-1])$,  Applying the functor $Hom_{\G}(V[n], \cdot)$ for any $n \in \mathbb{Z}$ and any $\G$-representation $V$, we get:
$$Hom_{\G}(V[n], A) \cong H^{n}(A \otimes V)\law 0 \raw \cong Ext^{n}_{\gamma}(\mathbb{Q}, V) \cong Hom_{\G}(V[n], \wedge^{*}(\gamma[-1])).$$
This implies that, under the level of the $\G$-representations, the above map is a quasi-isomorphism. Therefore $A$ is $1$-minimal.
\end{proof}
\begin{cor} \label{Main structure for coh connected cdga}
Let $A$ be a cohomologically connected cdga over $\G$. Then
\begin{itemize}
\item
There is a functor:
$$\rho: D^{b}(co-rep^{\G, f}_{\mathbb{Q}}(\chi_{A})) \longrightarrow \mathcal{D}^{\G, f}_{A}.$$
Furthermore, $\rho$ induces a functor on the hearts:
$$\mathcal{H}(\rho): co-rep^{\G, f}(\chi_{A}) \rightarrow \mathcal{H}^{\G, f}_{A},$$
which is an equivalence of Tannakian categories.
\item
The functor $\rho$ is an equivalence of triangulated categories if and only if $A$ is $1$-minimal.
\end{itemize}
\end{cor}

\section{Relation with the weighted completion} \label{relation}
Suppose that $\Gamma$ is an abstract group, that $G$ as before and that $\rho: \Gamma \to G(\mb{Q})$ is a homomorphism with Zariski dense image. We recall the definition of the weighted completion of $\Gamma$ relative to $\rho$ and a central cocharacter $\omega: \mb{G}_m \to R$ in \cite{HM1}. 
\begin{df}
A weighted $\Gamma$-module with respect to $\rho$ and $w$ is a finite dimensional $\mb{Q}$ vector space with $\Gamma$-action together with a weight filtration
$$\cdots \subset W_{m-1}M \subset W_{m}M \subset W_{m+1}M \subset \cdots$$
by $\Gamma$ invariant subspaces. And satisfy:
\begin{itemize}
\item
the intersection of the $W_{m}M$ is $0$ and their union is $M$,
\item
for each $m$, the representation $\Gamma \to Aut\ Gr^{W}_{m}M$ should factor through $\rho$ and a homomorphism $\phi_{m}: G \rightarrow Aut\ Gr^{W}_{m}M$,
\item
$Gr^{W}_{m}M$ has weight $m$ viewed as a $\mb{G}_m$-module via
$$\mb{G}_m \xrightarrow{w} G \xrightarrow{\phi_{m}} Aut \ Gr^{W}_{m}M. $$
\end{itemize}
\end{df}
\begin{rmk}
Weighted $\Gamma$-modules together with the $\Gamma$-equivariant morphisms between weighted $\Gamma$-modules form the category of weighted $\Gamma$-modules, denoted by $W-Mod^G(\Gamma)$. Notice that the category of weighted $\Gamma$-modules is Tannakian.
\end{rmk}
\begin{df} \label{df of weighted completion}
The weighted completion of $\Gamma$ with respect to $\rho$ and $w$ is the tannakian fundamental group of the category of weighted $\Gamma$-modules with respect to $\rho$ and $w$. We denote it by $\ml{G}$ for simplicity. Then we have a short exact sequence:
$$1 \to \ml{U} \to \ml{G} \to G \to 1,$$
where $\ml{U}$ is the prounipotent radical of $\ml{G}$. We denote the Lie algebra of $\ml{U}$ by $\mf{u}$.
\end{df}
Notice that, given $V$ any rational $G$-representation, we can view $V$ as a right $\Gamma$-module by $\rho$. In the following, we assume that $H^i(\Gamma, V)$ is finite dimensional for any finite dimensional rational $G$-representation and $i \in \mb{Z}_{\geq 0}$.
Let $E^{\bullet}$ be any cohomologically connected cdga over $G$. Applying Theorem \ref{existence of t-structure}, we get the heart $\ml{H}^{G, f}_{E^{\bullet}}$ of $\ml{D}^{G, f}_{E^{\bullet}}$. Because we are interested in the abelian category only, we may also assume that $E^{\bullet}$ is $1$-minimal. There is a canonical functor $\Phi$ from $\ml{H}^{G, f}_{E^{\bullet}}$ to $W-Mod^G(\Gamma)$ by sending $M \in \ml{H}^{G, f}_{E^{\bullet}}$ to itself. We denote the forgetful functor from $\ml{H}^{G, f}_{E^{\bullet}}$ (resp. $W-Mod^G(\Gamma)$) to the category of $\Gamma$ modules by $F_H$ (resp. $F_W$). Notice that the composition of $\Phi$ with $F_W$ is $F_H$.
\begin{rmk}
There is another equivalent definition of the weighted completion by universal properties. See Section 7 in \cite{HM1} for example. The above canonical functor is just induced by the canonical map $\ml{G} \to \pi_1(\ml{H}^{G, f}_{E^{\bullet}}, F)$, where $\pi_1(\ml{H}^{G, f}_{E^{\bullet}}, F)$ is the tannakian fundamental group of $\ml{H}^{G, f}_{E^{\bullet}}$ and $F$ is the forgetful functor to the category of $\mb{Q}$-vector spaces.
\end{rmk}
\textbf{Assumption:} For $i \in \mb{Z}_{> 0}$, there exists a map:
\begin{equation} \label{extension1}
\varphi_i: \bigoplus_{\alpha} H^{i}(\Gamma, (V_{\alpha})^{\vee}) \otimes V_{\alpha}  \to H^{i}(E^{\bullet}),
\end{equation}
where the index set runs through all irreducible representations with non-positive weights. Furthermore, we require that $\{ \varphi_i \}_{i \in \mb{Z}_{\geq 0}}$ is multiplicative, in other words,
the multiplicative  structure:
$$H^{i}(\Gamma, (V_{\alpha})^{\vee}) \otimes H^{j}(\Gamma, (V_{\beta})^{\vee}) \to H^{i+j}(\Gamma, (V_{\alpha} \otimes V_{\beta})^{\vee})$$
is compatible with the multiplicative structure on $H^{*}(E^{\bullet})$ induced by the multiplication of $E^{\bullet}$. 
\

In the next step, we want to use maps (\ref{extension1}) to construct a functor $\Psi: W-Mod^G(\Gamma) \to \ml{H}^{G, f}_{E^{\bullet}}$. In fact, after rewriting the maps  (\ref{extension1}) and using $1$-minimal property of $E^{\bullet}$, we get the following maps:
$$Ext^i_{\mb{Q}[\Gamma]-Mod}(\mb{Q}, (V_{\alpha})^{\vee}) \xrightarrow{\varphi} Ext^i_{\ml{H}^{G, f}_{E^{\bullet}}}(E^{\bullet}, E^{\bullet} \otimes (V_{\alpha})^{\vee}).$$
Furthermore, we have maps:
\begin{equation} \label{extension2}
Ext^i_{\mb{Q}[\Gamma]-Mod}(V_1, V_2) \to Ext^i_{\ml{H}^{G, f}_{E^{\bullet}}}(E^{\bullet} \otimes V_1, E^{\bullet} \otimes V_2),
\end{equation}
for any rational representations $V_1, V_2$ over $G$ such that $V_1 \otimes (V_2)^{\vee}$ is a representation with non-positive weight.

\begin{df}
For any $M \in W-Mod^G(\Gamma)$, we let $l_M$ be the integer such that $W_{l_M}M \neq 0$ and $W_{i}M = 0$ for $i < l_M$ and we let $u_M$ be the least integer such that $W_{u_m} M = M$. Then we define the length of $M$ to be $u_M - l_M + 1$ and denote it by $l(M)$.  
\end{df}

If $M \in W-Mod^G(\Gamma)$ has length one,  then $M$ is just a rational $G$-representation with a pure weight. We define $\Psi(M) = E^{\bullet} \otimes M$. In general, if $M \in W-Mod^G(\Gamma)$ has length bigger than one, then $M$ can be considered as an element $\xi$ in $Ext^{n-1}_{\mb{Q}[\Gamma]-Mod}(Gr^W_{u_M}M, Gr^W_{l_M}M)$. Then we define $\Psi(M)$ to be the element in $\ml{H}^{G, f}_{E^{\bullet}}$ corresponding to the image of $\xi$ under maps (\ref{extension2}), which is a $(n-1)$-step Yoneda extension of $E^{\bullet} \otimes Gr^W_{l_M}M$ by $E^{\bullet} \otimes Gr^W_{u_M}M$. 

\begin{thm}
Let $E^{\bullet}$ be a cdga over $G$ which satisfies the above assumption. In maps (\ref{extension1}), we assume further that $\varphi_1$ is an $G$-equivariant isomorphism and $\varphi_2$ is an $G$-equivariant injection. Then $\Psi$ is an equivalence between Tannakian categories. 
\end{thm}
\begin{proof}
Recall in section \ref{heart}, we could identify $\ml{H}^{G, f}_{E^{\bullet}}$ with the category of $(\gamma_{E^{\bullet}})^{\vee}$-representations in the category of rational $G$-representations. Here $(\gamma_{E^{\bullet}})^{\vee}$ is the dual of $\gamma_{E^{\bullet}}$ (Definition \ref{gamma}), i.e., a Lie algebra in the category of rational $G$-representations. Similarly, $W-Mod^G(\Gamma)$ can be identified with  the category of representations over $\mf{u}$ in the category of rational $G$-representations. Then $\Psi$ induces a $G$-equivariant morphism between pronilpotent Lie algebras $\varphi: (\gamma_{E^{\bullet}})^{\vee} \to \mf{u}$.

The requirement of $\varphi_1, \varphi_2$ will imply that: there exists a $G$-equivariant isomorphism:
\begin{equation}
\bigoplus_{\alpha} H^1(\Gamma, V_{\alpha}) \otimes (V_{\alpha})^{\vee} \cong H^{1}((\gamma_{E^{\bullet}})^{\vee}) .
\end{equation}
and a $G$-equivariant injection:
\begin{equation}
\bigoplus_{\alpha} H^2(\Gamma, V_{\alpha}) \otimes (V_{\alpha})^{\vee} \hookrightarrow H^{2}((\gamma_{E^{\bullet}})^{\vee}) .
\end{equation}
where the index set $\alpha$ runs all the irreducible $G$- representations with positive Adams weights.
\

Using Corollary 8.2. in \cite{HM1}, we get that:
\begin{enumerate}
\item
There is a natural $G$-equivariant isomorphism:
$$\varphi^*: H^{1}(\mf{u}) \xrightarrow{\cong} H^{1}((\gamma_{E^{\bullet}})^{\vee}).$$
\item
There is a natural $G$-equivariant injection:
$$\varphi^*: H^{2}(\mf{u}) \hookrightarrow H^{2}((\gamma_{E^{\bullet}})^{\vee}).$$
\end{enumerate}
Then by Proposition 2.1. in \cite{H2}, we know that $\varphi: (\gamma_{E^{\bullet}})^{\vee} \to \mf{u}$ is an isomorphism of Lie algebras in the category of $G$-represenations. Therefore, the functor $\Psi$ from $\ml{H}^{G, f}_{E^{\bullet}}$ to $W-Mod^G(\Gamma)$ is an equivalence.
\end{proof}

\section{Split Tannakian categories over a reductive group} \label{sp tan}
In this section, we describe a special kind of Tannakian categories. Let $R$ be any reductive group over $\mb{Q}$.
\begin{df}
We say that $C$ is a neutral Tannakian category over $R$ if $C$ is a neutral Tannakian category over $\mb{Q}$ and there exists an exact faithful $\mb{Q}$-linear tensor functor $\widetilde{\omega}: C \to Rep(R)$, whose composition with the forgetful functor $F: Rep(R) \to Vec_{\mb{Q}}$ is the fiber functor $\omega: C \to Vec_\mb{Q}$.
\end{df}
\begin{exe} \label{example of main}
Assume that $R$ satisfies Convection \ref{conv} and that $A$ is a cohomologically connected cdga over $R$. From Theorem \ref{existence of t-structure}, we know the existence of the heart $\ml{H}^{R, f}_A$ of $\ml{D}^{R, f}_A$, which is a neutral Tannakian category over $R$.
\end{exe}
\begin{exe} \label{example of stc}
Another typical example is the relative completion $\ml{G}$ with respect a given map from a finite generated group $\Gamma \xrightarrow{\rho} R(\mb{Q})$. For the definition of the relative completion, we refer to \cite{H}. Roughly speaking, the relative completion is just the weighted completion without considering the weight. 
In \cite{H2}, Hain shows that the relative completion $\ml{G}$ can be expressed as a semi-product of a prounipotent algebraic group and $R$. This implies that there exists a functor from the category of finite dimensional representations over $\ml{G}$, which is a Tannakian category, to the category of rational representations over $R$. It's easy to check this functor is an exact faithful tensor functor. 
\end{exe}
Motivated by the above examples, we define:
\begin{df} \label{split tan}
A neutral Tannakian category $C$ over $R$ is split if the full subcategory of $C$ consisting of semi-simple objects is isomorphic to $Rep(R)$.
\end{df} 
\begin{rmk} \label{tann fun}
The Tannakian fundamental group of a neutral Tannakian category $C$ with a tensor generator is isomorphic to a linear proalgebraic group. See Proposition 2.20 in \cite{DM}. If we assume further that this Tannakian category is split over $G$, then its Tannakian fundamental group will be the form $U \rtimes R$, where $U$ is a prounipotent algebraic group. Example \ref{example of main} and Example \ref{example of stc} satisfy these conditions.
\end{rmk}
In the end, we want to use the method of framed objects (Section 6 of \cite{BK} for example) to give a description of the coordinate ring of the Tannakian fundamental group of any split neautral Tannkian category $C$ with a tensor generator. As explained in Remark \ref{tann fun}, it's enough to determine the coordinate ring of the prounipotent radical as a Hopf algebra object in $Rep(R)$.
\begin{df}
A framed object in $C$ is an object $X$ in $C$ together with an element $u \in Hom_{R}(V, X)$ (called the frame vector) and an element of $v \in Hom_{R}(X , \mb{Q})$ (called the frame covector), where $V$ is an irreducible $G$-representation.  We denote this object by $(X, u, v)$.
\end{df}
\begin{df}
Two framed objects $X, Y \in C$ are identified if there is a mapping $X \to Y$ compatible with the framing.
\end{df}
Notice that such pairs $X, Y$ define a relation $\ml{R}$ on the set of all framed objects. Then $\chi_C$ is defined as the set of equivalence classes of the smallest equivalence relation containing $\ml{R}$. By definition, $\chi_C$ is graded over the irreducible rational representations of $G$. For a given rational $R$- representation $V$, we denote the $V$ graded piece of $\chi_C$ by $\chi_C(V)$.
\begin{cl} \label{claim}
$\chi_C$ is a Hopf algebra in $Rep(R)$.
\end{cl}
In order to explain our claim, first we recall a fundamental result -- Proposition 3.1 in \cite{H}. If $(V_{\alpha})_{\alpha}$ is a set of irreducible rational $R$-representations, viewed right $R$-modules, then, as an $(R, R)$ bimodule, $\ml{O}(R)$ is canonically isomorphic to $\bigoplus_{\alpha} (V_{\alpha})^{\vee} \boxtimes V_{\alpha}.$ In other words, there exists a Hopf algebraic structure on $\bigoplus_{\alpha} (V_{\alpha})^{\vee} \boxtimes V_{\alpha}.$ If we denote the generator of $(V_{\alpha})^{\vee} \boxtimes V_{\alpha}$ by $v_{\alpha}^{\vee} \boxtimes v_{\alpha}$, then we have:
\begin{enumerate}
\item[(1).]
$m((v_{\alpha}^{\vee} \boxtimes v_{\alpha}) \otimes (v_{\beta}^{\vee} \boxtimes v_{\beta})) = \oplus_{\gamma} (v_{\gamma}^{\vee} \boxtimes v_{\gamma})$, where $m$ is the multiplication of the Hopf algebra and the index set runs through the irreducible representations $V_{\gamma}$ appearing in the tensor product of $V_{\alpha} \otimes V_{\beta}$.
\item[(2).] \label{coproduct}
Let $\Delta$ be the coproduct map of the Hopf algebra $\bigoplus_{\alpha} (V_{\alpha})^{\vee} \boxtimes V_{\alpha}.$ Then we can express $$\Delta(v_{\alpha}^{\vee} \boxtimes v_{\alpha}) = \oplus_{\beta, \gamma} m^{\beta, \gamma}_{\alpha}((v_{\beta})^{\vee} \boxtimes v_{\beta}) \otimes  ((v_{\gamma})^{\vee} \boxtimes v_{\gamma}),$$
where $m^{\beta, \gamma}_{\alpha}$ is the corresponding multiplicity.
\end{enumerate}
Now we move to our claim.
\begin{itemize}
\item
The sum on $\chi_C$ is the Baer sum. 
\item
The product is defined by the tensor product of underlying objects together with the tensor product of the framings. Let $(X_1, u_1, v_1), (X_2, u_2, V_2)$ be two framed objects, which represent two classes $[(X_1, u_1, v_1)]$ and $ [(X_2, u_2, v_2)]$ in $\chi_C(V_{\alpha})$ and $\chi_C(V_{\beta})$ respectively. Then we define:
\begin{equation}
[(X_1, u_1, v_1)] \cdot [(X_2, u_2, v_2)] = [(X_1 \otimes X_2, v_1 \otimes v_2, v_1 \otimes v_2)] \in \chi_C(V_{\alpha} \otimes V_{\beta}),
\end{equation}
Then one can project this element into pieces $\chi_C(V_{\gamma})$, where  the irreducible representations $V_{\gamma}$ appearing in the tensor product of $V_{\alpha} \otimes V_{\beta}$.
\item
The coproduct 
\begin{equation}
\begin{split}
&\psi = \oplus_{\alpha} \psi_{\alpha}; \\
&\psi_{\alpha} = \oplus  \psi_{\alpha}^{\beta, \gamma};\\
& \psi^{\beta, \gamma}_{\alpha}: \chi_C(V_{\alpha}) \to \chi_C(V_{\beta}) \otimes \chi_C(V_{\gamma}),
\end{split}
\end{equation} 
where the first index set $\alpha$ runs over all irreducible representations $V_{\alpha}$ and the second index set is the same as the index set appearing in the coproduct (\ref{coproduct}), is defined as follows. We let $[(M, u, v)] \in \chi_C(V_{\alpha})$ and let
\begin{equation}
\sum x_i \otimes y_i = 1 \in Hom_R(H, V_{\gamma}) \otimes Hom_R(V_{\gamma}, H).
\end{equation}
Here $1$ corresponds the identity map under the isomorphism: $$Hom_R(H, V_{\gamma}) \otimes Hom_R(V_{\gamma}, H) \cong End(Hom_R(V_{\gamma}, H)).$$ Let $M_{x_i}$ be $M$ with frame vector $u$ and frame covector $x_i$ and let $M_{y_i}$ be $M$ with frame vector $y_i$ and frame covector $v$. Then we put
\begin{equation}
\psi^{\beta, \gamma}_{\alpha}([(M, u, v)]) = \sum [(M, u, x_i)] \otimes [(M, y_i, v)].
\end{equation}
\end{itemize}
\begin{prop} \label{framed objects}
Let $C$ be a split neutral Tannakian category with a tensor generator. Then $C$ is equivalent to the category of finite dimensional $\chi_C$-comodules in $Rep(R)$.
\end{prop}
\begin{proof}
Step 1.  As we state in Remark \ref{tann fun}, $C$ is isomorphic to the category of representations over a linear proalgebraic group $\ml{G}$, which is a semi-product of a prounipotent algebraic group $\ml{U}$ with $R$. Recall the following basic properties about extension groups in $C$ or the category of finite dimensional $\ml{G}$-representations (Section 5.1 in \cite{HM1}):
\begin{enumerate}
\item[A.] 
The category of $\ml{G}$-modules has enough injectives, therefore the cohomology groups can be identified with the extension groups:
$$H^i(\ml{G}, V) = Ext^i_{\ml{G}-mod}(\mb{Q}, V)$$
for $V \in \ml{G}-mod$ and $\mb{Q}$ is the trivial object or the trivial representation. If $V$ is the trivial representation $\mb{Q}$, we will denote $H^i(\ml{G}, V) $ by $H^i(\ml{G})$ for short. When $V \in C$, we have:
$$H^i(\ml{G}, V) = Ext^i_{C}(\mb{Q}, V),$$
because each element in the extension class can be replaces by an equivalent extension consisting of only finite dimensional representations.
\item[B.]
We consider $V \in Rep(R)$ as a $\ml{G}$-representation. Then we have isomorphism in $Rep(R)$ (Theorem 5.3 in \cite{HM1}):
\begin{equation}  \label{q}
H^i(\ml{U}) \cong  \bigoplus_{\alpha} H^i(\ml{G}, (V_{\alpha})^{\vee}) \otimes V_{\alpha} \cong  \bigoplus_{\alpha} Ext^i_C(1, (V_{\alpha})^{\vee}) \otimes V_{\alpha},
\end{equation} 
where the index set runs over all irreducible $R$-representations.
\end{enumerate}

Step 2. Notice that the coproduct on $\chi_C$ is conilpotent. For the definition of a conilpotent coproduct, we refer to Section 3.8 in \cite{C}. Using Theorem 3.9.1 in \cite{C}, we know that $\ml{U}_C$, which is defined to be $Spec(\chi_C)$, is a prounipotent algebraic group over $\mb{Q}$.
\

\noindent Step 3. Now we want to construct a functor from $C$, which can be identified as the category of representations over $\ml{G}$, to the category of finite dimensional $\chi_C$-comodules (the same as $\ml{U}_C$ modules) in $Rep(R)$. Because $Rep(R)$ is a full subcategory of $C$ as we assume, we also consider $V \in Rep(R)$ as an element in $C$. Let $M \in C$. Notice that $M \cong \bigoplus_{\alpha} Hom_C(V_{\alpha}, M) \otimes V_{\alpha}$. Under this isomorphism, any $m \in M$ can be decomposed as the sum of the form $u_{\alpha} \otimes m_{\alpha} \in Hom_R(V_{\alpha}, M) \otimes V_{\alpha}$. Then we define a $\chi_C$-comodule structure on $M$ by:
\begin{equation}
M \xrightarrow{\varphi_M} \chi_C \otimes M: m \to \oplus_{\alpha} [(M, u_\alpha, 0)] \otimes m_{\alpha}.
\end{equation} 
We denote the functor from $C$ to the category of $\chi_C$-comodules in $Rep(R)$ by $\Phi$. Then under Tannakian duality, $\Phi$ induces a group homomorphism from $\ml{U}_C$ to $\ml{U}$.
\

\noindent Step 4. By the definition of $\chi_C$, we may check:
\begin{itemize} 
\item
$H^1(\ml{U}_C) \cong \bigoplus_{\alpha} Ext^1_C(1, (V_{\alpha})^{\vee}) \otimes V_{\alpha}$ and $H^2(\ml{U}_C) \cong \bigoplus_{\alpha} Ext^2_C(1, (V_{\alpha})^{\vee}) \otimes V_{\alpha}.$
\item
Furthermore, $\Phi$ induces a map from $H^i(\ml{U}) \to H^i(\ml{U}_{C})$, which are coincide with (\ref{q}).
\end{itemize}
If we denote the corresponding Lie algebra of $\ml{U}_C$ and $\ml{U}$ by $\mf{u}_C$ and $\mf{u}$, using Proposition 2.1 in \cite{HM1} again, we know that: $\Phi: \mf{u}_C \to \mf{u}$ is a $R$-equivariant isomorphism between pronilpotent Lie algebras. Therefore, $\Phi$ induces an equivalence of Tannkian categories between $C$ and the category of finite dimensional $\chi_C$-comodules in $Rep(R)$.
\end{proof}
\begin{exe}
Let $F$ be a field finitely generated over a prime field and let $l$ be a prime number. A mixed Tate representation of the absolute Galois group $Gal(\bar{F}/F)$ over $F$ is a filtered representation such that the odd graded pieces are $0$ and the $(2r)$-th graded piece is a sum of copies of $\mb{Q}_l(-r)$. Then the category of mixed Tate representations of $Gal(\bar{F}/F)$, denoted by $MTM_{F, l}$, is a split neutral Tannakian category over $\mb{G}_m$. Using Proposition \ref{framed objects}, we know that $MTM_{l, F}$ is isomorphic to the category of finite dimensional $\chi_{MTM_{l, F}}$ comodules in $Rep(\mb{G}_m)$. The latter fact has been shown in \cite{BGSV}.
\end{exe}
\begin{exe}
Let $F$ be a number field, then the abelian category $MTM(F, \mb{Q})$ of mixed Tate motives with rational coefficients over $F$ exists. See \cite{L} for example. In fact,  $MTM(F, \mb{Q})$ is also a split neutral Tannakian category over $\mb{G}_m$. One may use Proposition \ref{framed objects} to describe such category. Moreover, the full rigid tensor subcategory of the abelian category of mixed motives generated by the motive of a fixed smooth projective variety is conjecturally a split neutral Tannakian categories over some reductive group. 
\end{exe}

\bigskip

Jin Cao, Yau Mathematical Sciences Center, Tsinghua University, Beijing, China,
\

\textit{E-mail address}: jcao@math.tsinghua.edu.cn

\end{document}